\documentclass[a4paper,11pt]{article}
\usepackage{amsmath}
\usepackage{amssymb}
\usepackage{amsthm}
\usepackage{MnSymbol}
\usepackage{paralist}
\usepackage{enumitem}
\usepackage{hyperref,url}
\usepackage{bm} 
\usepackage{placeins}
\usepackage{color}
\usepackage[margin=30mm]{geometry}

\newcommand{\N}{\mathbb{N}}
\newcommand{\R}{\mathbb{R}}
\newcommand{\Z}{\mathbb{Z}}
\renewcommand{\div}{ \mathrm{div}  }

\newcommand{\dHaus}{\, d\Gamma}
\newcommand{\dd}{\, d}
\newcommand{\dx}{\, dx}
\newcommand{\dt}{\, dt}

\newcommand{\pd}{\partial}
\newcommand{\pdnu}{\pd_{\bm{\nu}}}

\renewcommand{\AA}{\underline{\underline{\mathrm{A}}}}

\newcommand{\abs}[1]{\left| #1 \right|}
\newcommand{\norm}[1]{\| #1 \|}

\newcommand{\inner}[2]{\langle #1 , #2 \rangle}

\newcommand{\eps}{\varepsilon}
\newcommand{\Lap}{\Delta}
\newcommand{\surf}{\nabla_{\Gamma}}
\newcommand{\LB}{\Delta_{\Gamma}}

\theoremstyle{plain}
\newtheorem{thm}{Theorem}[section]

\newtheorem{lemma}[thm]{Lemma}

\theoremstyle{plain}
\newtheorem{remark}{Remark}[section]

\numberwithin{equation}{section}

\begin{document}

\title{On a coupled bulk-surface Allen--Cahn system with an affine linear transmission condition and its approximation by a Robin boundary condition}

\author{Pierluigi Colli \footnotemark[1] \and Takeshi Fukao \footnotemark[2] \and Kei Fong Lam \footnotemark[3]}

\date{ }

\maketitle

\renewcommand{\thefootnote}{\fnsymbol{footnote}}
\footnotetext[1]{Dipartimento di Matematica, Universit\`{a} degli Studi di Pavia, Via Ferrata 5, 27100 Pavia, Italy \tt (pierluigi.colli@unipv.it).}
\footnotetext[2]{Department of Mathematics, Faculty of Education, Kyoto University of Education, 1 Fujinomori, Fukakusa, Fushimi-ku, Kyoto 612-8522 Japan \tt (fukao@kyokyo-u.ac.jp).}
\footnotetext[3]{Department of Mathematics, The Chinese University of Hong Kong, Shatin, N.T., Hong Kong \tt (kflam@math.cuhk.edu.hk)}

\begin{abstract}
We study a coupled bulk-surface Allen--Cahn system with an affine linear transmission condition, that is, the trace values of the bulk variable and the values of the surface variable are connected via an affine relation, and this serves to generalize the usual dynamic boundary conditions.  We tackle the problem of well-posedness via a penalization method using Robin boundary conditions.  In particular, for the relaxation problem, the strong well-posedness and long-time behavior of solutions can be shown for more general and possibly nonlinear relations.  New difficulties arise since the surface variable is no longer the trace of the bulk variable, and uniform estimates in the relaxation parameter are scarce.  Nevertheless, weak convergence to the original problem with affine linear relations can be shown.  Using the approach of Colli and Fukao (Math. Models Appl. Sci. 2015), we show strong existence to the original problem with affine linear relations, and derive an error estimate between solutions to the relaxed and original problems.
\end{abstract}



\section{Introduction}\label{sec:Intro}
For parabolic partial differential equations the most common boundary conditions encountered will be of Dirichlet, Neumann or Robin type, which can be roughly classified as spatial in nature, that is, no time derivatives are present in the boundary conditions.  Meanwhile, boundary conditions involving the time derivative of the variables are called dynamic boundary conditions, and this may arise when one considers the effects of the boundary of the domain on the evolution of the variable of interest.  For instance, a prototype heat equation with dynamic boundary conditions posed in a bounded domain $\Omega$ with boundary $\Gamma := \pd \Omega$ reads as
\begin{align*}
\pd_{t} u = \Lap u + f & \text{ in } \Omega, \\
\gamma \pd_{t} u_{\Gamma} = \sigma \LB u_{\Gamma} - \pdnu u - \kappa u_{\Gamma} + g & \text{ on } \Gamma, \\
u_{\Gamma} = u \vert_{\Gamma} & \text{ on } \Gamma,
\end{align*}
for some non-negative constants $\gamma$, $\sigma$ and $\kappa$, and prescribed data $f$ and $g$.  In the above, $\pdnu u := \nabla u \cdot \bm{\nu}$ denotes the normal derivative of $u$ on $\Gamma$, with unit outer normal $\bm{\nu}$, and $\LB$ denotes the Laplace--Beltrami operator on $\Gamma$ (see for example \cite[\S 2.2, (2.5)]{DE}).  The presence of the Laplace--Beltrami operator signals the existence of a \emph{surface free energy} for the surface variable $u_{\Gamma}$, which is the trace of $u$ on $\Gamma$ according to the third equation. 

For equations arising from diffuse interface models, such as the Allen--Cahn equation \cite{AC}, the Cahn--Hilliard equation \cite{CH} and the Caginalp system \cite{Cag}, the problems with dynamic boundary conditions have been studied by many authors (see for example \cite{CalaColli, ColliFukaoAC, CGNS, GalG1, GalG2, Gilardi, Miranville, WuZheng} and the references cited therein).  Moreover, one can also include the effects of phase separation on the boundary $\Gamma$ by including additional terms in the dynamic boundary conditions that lead to Allen--Cahn or Cahn--Hilliard structures \cite{Fischer,Kenzler} (see also \cite{LiuWu} and the references cited therein).  However, the transmission condition $u_{\Gamma} = u \vert_{\Gamma}$ between the bulk and surface variables remains unchanged.

In this work, we consider a modification of the transmission condition, namely by denoting the bulk variable as $u$ and the surface variable as $\phi$, we ask that $u \vert_{\Gamma} = h(\phi)$ holds on $\Gamma$ for some continuous function $h : \R \to \R$.  The usual transmission condition is recovered by setting $h$ as the identity function, i.e., $h(s) = s$.  Our motivation for this choice of transmission condition arises from the special case $h(s) = -s$.  In the context of phase separation, the surface phase is the opposite of the bulk phase, and this may lead to new and interesting couplings between bulk and surface dynamics. 

For instance, consider a  bulk Cahn--Hilliard equation for $u$ coupled to a surface Cahn--Hilliard equation for $\phi$ via the transmission condition $u \vert_{\Gamma} = - \phi$ on the boundary $\Gamma$.  It is well-known that for a constant mobility, the bulk Cahn--Hilliard equation approximates the Mullins--Sekerka flow \cite{Escher, MS} in the limit of vanishing interfacial thickness \cite{ABC, Chen, Pego}, and the corresponding approximation of a surface Mullins--Sekerka flow also holds (at least formally) for the surface Cahn--Hilliard equation (see \cite{Bjorn} which treats a more general setting involving evolving surfaces, or \cite{GKRR} which treats a coupled surface-Cahn--Hilliard bulk-diffusion model).  It is also well-known that the Mullins--Sekerka flow can be seen as a nonlocal motion by mean curvature that preserves volume \cite[\S 5]{Pego} and that $d$-spheres (for the problem posed in $\R^{d}$) are exponentially stable \cite{Escher}.  Then, we expect that in the case $h(s) = s$, i.e., $u \vert_{\Gamma} = \phi$, a ball $B$ of phase $+1$ located in the interior of the domain $\Omega$ and surrounded by a sea of phase $-1$ (i.e., $\Omega \setminus \overline{B}$ and $\Gamma$ are of phase $-1$) would be a stable configuration.  However, in the case $h(s) = -s$, due to the mismatch of the boundary values, the transmission condition would attempt to change the values on the boundary, and we expect that for relaxation dynamics the ball of phase $+1$ may shrink while the phase $+1$ would start to appear at the boundary $\Gamma$.  If the volume of the ball is equal to the surface area of $\Gamma$, then as a stable configuration we would have only phase $-1$ in $\Omega$ and only phase $+1$ on $\Gamma$.

We begin our investigation with the simplest case: a coupled bulk-surface Allen--Cahn system of the form
\begin{subequations}\label{Intro:ACAC:lim}
\begin{alignat}{3}
\pd_{t} u - \Lap u + \beta(u) + \pi(u) \ni f & \text{ in } Q := \Omega \times (0,T), \\
\pd_{t}\phi - \LB \phi + \beta_{\Gamma}(\phi) + \pi_{\Gamma}(\phi) + h'(\phi) \pdnu u \ni f_{\Gamma} & \text{ on } \Sigma := \Gamma \times (0,T), \\
u = h(\phi) & \text{ on } \Sigma,
\end{alignat}
\end{subequations}
where $\Omega \subset \R^{3}$ is a bounded domain and $T > 0$ is an arbitrary but fixed constant.  In the above, $f$ and $f_{\Gamma}$ are prescribed external forcings, $\beta$ and $\beta_{\Gamma}$ are maximal monotone and possibly non-smooth graphs, while $\pi$ and $\pi_{\Gamma}$ are non-monotone Lipschitz perturbations.  Let us point out the effect of a non-trivial relation $u = h(\phi)$ leads to the appearance of a prefactor $h'(\phi)$ multiplied with the normal derivative $\pdnu u$.  In the case $h(s) = s$, this reduces to $h'(\phi) = 1$ and we are in the setting of Calatroni and Colli \cite{CalaColli}.  If we define $g$ as the inverse of $h$, i.e., $\phi = g(u)$, then by the chain rule $h'(\phi) = (g'(u))^{-1}$, and we can reformulate \eqref{Intro:ACAC:lim} into the system
\begin{subequations}\label{Intro:ACAC:lim:g}
\begin{alignat}{3}
\pd_{t} u - \Lap u + \beta(u) + \pi(u) \ni f & \text{ in } Q, \\
g'(u_{\Gamma}) \left ( \pd_{t} g(u_{\Gamma}) - \LB g(u_{\Gamma}) + \beta_{\Gamma}(g(u_{\Gamma})) + \pi_{\Gamma}(g(u_{\Gamma})) - f_{\Gamma} \right ) + \pdnu u \ni 0 & \text{ on } \Sigma, \\
u \vert_{\Gamma} = u_{\Gamma} & \text{ on } \Sigma,
\end{alignat}
\end{subequations}
whose variational formulation reads as
\begin{equation}\label{Weakform:limit:g}
\begin{aligned}
0 & = \int_{\Omega} \left ( \pd_{t} u + \xi + \pi(u) - f \right ) \, \zeta + \nabla u \cdot \nabla \zeta  \dx \\
& \quad + \int_{\Gamma} g'(u_{\Gamma}) \left ( \pd_{t} g(u_{\Gamma}) + \xi_{\Gamma} + \pi_{\Gamma}(g(u_{\Gamma})) - f_{\Gamma} \right ) \, \zeta_{\Gamma} + \surf g(u_{\Gamma}) \cdot \surf \left ( g'(u_{\Gamma}) \zeta_{\Gamma} \right ) \dHaus
\end{aligned}
\end{equation}
for all $\zeta \in H^{1}(\Omega)$ such that $\zeta_{\Gamma} := \zeta \vert_{\Gamma} \in H^{1}(\Gamma)$ and $\xi \in \beta(u)$ a.e.~in $Q$, $\xi_{\Gamma} \in \beta_{\Gamma}(g(u_{\Gamma}))$ a.e.~on $\Sigma$.

Immediately one observes that there will be difficulties in passing to the limit in some approximation scheme to obtain the last term of \eqref{Weakform:limit:g} if $g$ is a nonlinear function.  In the case $g$ (and also $h$) is an affine linear function, i.e., $g(s) = \alpha^{-1}(s-\eta)$ for some $\alpha \neq 0$, $\eta \in \R$, we can appeal to the procedure in Colli and Fukao \cite{ColliFukaoAC} to deduce the existence of strong solutions to \eqref{Intro:ACAC:lim:g}.  However, for a more general and possibly nonlinear relation even the existence of a weak solution to \eqref{Intro:ACAC:lim:g} is an open problem due to the highly nonlinear surface equation.

Let us also mention that in \cite{CalaColli, ColliFukaoAC} the maximal monotone graphs $\beta$ and $\beta_{\Gamma}$ satisfy a compatibility condition \cite[(2.22)-(2.23)]{CalaColli}, which in some sense requires that $\beta_{\Gamma}$ is dominating $\beta$.  As discussed in Remark \ref{rem:CalaColli}, this type of assumption may not hold for the affine linear case, and thus we encounter new difficulties in deducing estimates for the selections $\xi$ and $\xi_{\Gamma}$.  This can be overcome by prescribing some growth assumptions on $\beta$ and $\beta_{\Gamma}$ such as \eqref{Limit:strong:beta}.

Alternatively, we can view $u \vert_{\Gamma} = h(\phi)$ as a Dirichlet boundary condition for $u$ and approximate it using a Robin boundary condition (also known as the boundary penalty method \cite{Babuska,Barrett}).  For $K > 0$ consider the system
\begin{subequations}\label{Intro:ACAC:gen}
\begin{alignat}{3}
\pd_{t} u_{K}  - \Lap u_{K} + \beta(u_{K}) + \pi(u_{K}) \ni f & \text{ in } Q, \\
\pd_{t} \phi_{K}  - \LB \phi_{K} + \beta_{\Gamma}(\phi_{K}) + \pi_{\Gamma}(\phi_{K}) + h'(\phi_{K}) \pdnu u_{K} \ni f_{\Gamma} & \text{ on } \Sigma, \\
K \pdnu u_{K} + u_{K} = h(\phi_{K}) & \text{ on } \Sigma,
\end{alignat}
\end{subequations}
where we now view $u_{K}$ and $\phi_{K}$ as independent variables that are coupled via the term $h'(\phi_{K}) \pdnu u_{K}$ and the Robin boundary condition.  In the formal limit $K \to 0$, we recover the transmission condition $u \vert_{\Gamma} = h(\phi)$.  It turns out we can prove strong well-posedness and long-time behavior of solutions to \eqref{Intro:ACAC:gen} for relations $h \in C^{2}(\R)$ that only need to satisfy $h', h'' \in L^{\infty}(\R)$.  Hence, nonlinear relations are possible for the relaxation problem \eqref{Intro:ACAC:gen}.  Furthermore, in the affine linear case $h(s) = \alpha s + \eta$, the sequence $\{(u_{K}, \phi_{K})\}_{K \in (0,1]}$ converges weakly to a limit $(u,\phi)$ as $K \to 0$ which is a weak solution to \eqref{Intro:ACAC:lim}.  Thanks to the strong well-posedness of \eqref{Intro:ACAC:lim} for affine linear relations we can also derive an error estimate of the form (under the same data and initial conditions)
\begin{align*}
\norm{u_{K} - u}_{\mathbb{X}_{\Omega}} + \norm{\phi_{K} - \phi}_{\mathbb{X}_{\Gamma}} + K^{-\frac{1}{2}} \norm{\alpha \phi_{K} + \eta - u_{K}}_{L^{2}(\Sigma)} \leq C K^{\frac{1}{2}} \norm{\pdnu u}_{L^{2}(\Sigma)},
\end{align*}
where $\mathbb{X}_{\Omega} := L^{\infty}(0,T;L^{2}(\Omega)) \cap L^{2}(0,T;H^{1}(\Omega))$ and $\mathbb{X}_{\Gamma}$ is defined similarly.  In particular, the transmission condition $u \vert_{\Sigma} = \alpha \phi + \eta$ on $\Sigma$ of the limit problem \eqref{Intro:ACAC:lim} is obtained from the Robin problem \eqref{Intro:ACAC:gen} at a linear rate in $K$.

Let us summarize on the main novelties of this work: we study the well-posedness of a coupled bulk-surface Allen--Cahn system with maximal monotone graphs and an affine linear transmission condition via two methods.  Strong solutions are established by using an abstract formulation, while weak solutions can be obtained as weak limits of a relaxation problem with Robin boundary conditions.  Strong well-posedness and long-time behavior for the relaxation problem with rather general relation $h$ are established, and a rate of convergence to strong solutions of the original problem is given. 

The structure of this article is as follows:  in Secection \ref{sec:derivation}, we derive the Robin approximation of coupled bulk-surface Allen--Cahn/Cahn--Hilliard systems that are thermodynamically consistent.  The main results on the coupled bulk-surface Allen--Cahn system are stated in Section \ref{sec:main}.  In Section \ref{sec:ctsdep} and \ref{sec:exist}, via a two-level approximation, the strong well-posdness of \eqref{Intro:ACAC:gen} is shown.  Long-time behavior of solutions is discussed in Section \ref{sec:longtime}, and in Section \ref{sec:limit} we study the well-posedness of \eqref{Intro:ACAC:lim} first by showing the weak convergence of solutions to \eqref{Intro:ACAC:gen}, and then by establishing strong solutions to \eqref{Intro:ACAC:lim:g}.  Lastly an error estimate between solutions of \eqref{Intro:ACAC:lim} and \eqref{Intro:ACAC:gen} is derived.

\section{Formal derivation}\label{sec:derivation}
Let $\Omega \subset \R^{3}$ denote a bounded domain with smooth boundary $\Gamma$.  For fixed time $T > 0$ let $Q := \Omega \times (0,T)$ and $\Sigma := \Gamma \times (0,T)$.  Let $u$ and $\phi$ denote variables that satisfy the following balance laws
\begin{equation}\label{prototype}
\begin{aligned}
\pd_{t} u + \div \bm{J}_{u} + R = 0 & \text{ in  } Q, \\
\pd_{t} \phi + \div_{\Gamma} \bm{J}_{\phi} + Z = 0 & \text{ on } \Sigma,
\end{aligned}
\end{equation}
where the fluxes $\bm{J}_{u}$, $\bm{J}_{\phi}$ and the reaction terms $R$ and $Z$ are yet to be determined.  We prescribe a free energy of the form
\begin{align}
\label{Derivation:Energy}
E(u, \phi) := \int_{\Omega} \frac{1}{2} \abs{\nabla u}^{2} + W(u) \dx + \int_{\Gamma} \frac{1}{2} \abs{\surf \phi}^{2} + W_\Gamma(\phi) \dHaus + \int_{\Gamma} \frac{1}{2K} \abs{u - h(\phi)}^{2} \dHaus,
\end{align}
where $K > 0$ is a constant.  The first and second terms of $E$ are the bulk and surface Ginzburg--Landau free energies, respectively, with potentials $W$ and $W_{\Gamma}$, and the third term measures the deviation of $u \vert_{\Gamma}$ from $h(\phi)$.  To derive thermodynamically consistent model equations based on \eqref{prototype}, we employ the Lagrange multiplier method \cite{Liu} of M\"{u}ller and Liu, see also \cite[Section 2.2]{AGG12} and \cite[Chapter 7]{book:Liu} for more details.  Let $\lambda_{u}$ and $\lambda_{\phi}$ denote Lagrange multipliers for \eqref{prototype}.  Then the procedure of M\"{u}ller and Liu is to enforce
\begin{align}
\label{dissipation}
\mathcal{D} := \frac{\dd E}{\dt} - \int_{\Omega} \lambda_{u} \left ( \pd_{t}u + \div \bm{J}_{u} + R \right ) \dx - \int_{\Gamma} \lambda_{\phi} \left ( \pd_{t} \phi + \div_{\Gamma} \bm{J}_{\phi} + Z \right ) \dHaus \leq 0
\end{align}
for arbitrary $u$, $\phi$, $\pd_{t} u$ and $\pd_{t} \phi$.  A short computation shows that
\begin{align*}
\mathcal{D} & = \int_{\Omega} (-\Lap u + W'(u) - \lambda_{u}) \pd_{t} u + \nabla \lambda_{u} \cdot \bm{J}_{u} - \lambda_{u} R \dx \\
& \quad + \int_{\Gamma} (\pdnu u + K^{-1}(u - h(\phi))) \pd_{t} u \dHaus - \int_{\Gamma} Z \lambda_{\phi} + \lambda_{u} \bm{J}_{u} \cdot \bm{\nu} \dHaus \\
& \quad + \int_{\Gamma} (-\LB \phi + W_{\Gamma}'(\phi) - K^{-1} (u - h(\phi))h'(\phi) - \lambda_{\phi}) \pd_{t} \phi + \surf \lambda_{\phi} \cdot \bm{J}_{\phi} \dHaus.
\end{align*}
In order for $\mathcal{D} \leq 0$ to be satisfied for arbitrary $u$, $\phi$, $\pd_{t} u$ and $\pd_{t} \phi$, the prefactors in front of $\pd_{t} u$ and $\pd_{t} \phi$ must vanish.  That is, we make the constitutive assumptions
\begin{equation}\label{constitutive}
\begin{aligned}
\lambda_{u} & = -\Lap u + W'(u), \\
\lambda_{\phi} & = -\LB \phi + W_{\Gamma}'(\phi) - K^{-1}(u - h(\phi))h'(\phi), \\
K \pdnu u & = h(\phi) - u.
\end{aligned}
\end{equation}
Different considerations of the fluxes $\bm{J}_{u}, \bm{J}_{\phi}$ and the reaction terms $R, Z$ will lead to different sets of equations, and now we will consider the following choices:

\subsection{Allen--Cahn/Allen--Cahn system} Setting $\bm{J}_{u} = \bm{0}$, $\bm{J}_{\phi} = \bm{0}$, and choosing 
\begin{align*}
R = \lambda_{u}, \quad Z = \lambda_{\phi}
\end{align*}
leads to a coupled bulk-surface Allen--Cahn system:
\begin{equation}\label{ACAC}
\begin{alignedat}{3}
\pd_{t}u  - \Lap u + W'(u) & = 0 & \text{ in } Q, \\
\pd_{t}\phi - \LB \phi + W_{\Gamma}'(\phi) + h'(\phi) \pdnu u & = 0 & \text{ on } \Sigma, \\
K \pdnu u + u & = h(\phi) \quad & \text{ on } \Sigma,
\end{alignedat}
\end{equation}
that satisfies the energy identity
\begin{align}\label{ACAC:Energy}
\frac{\dd}{\dt} E(u,\phi) + \int_{\Omega} \abs{\pd_{t}u}^{2} \dx + \int_{\Gamma} \abs{\pd_{t}\phi}^{2} \dHaus = 0.
\end{align}
\subsection{Cahn--Hilliard/Cahn--Hilliard system} 
Setting $R = 0$ and choosing for a positive constant $M$ and non-negative mobilities $m(u)$ and $n(\phi)$,
\begin{align*}
\bm{J}_{u} = - m(u) \nabla \lambda_{u}, \quad \bm{J}_{\phi} = - n(\phi) \surf \lambda_{\phi}, \quad Z = m(u)\pdnu \lambda_{u} = M^{-1}(\lambda_{\phi} - \lambda_{u}),
\end{align*}
leads to a coupled bulk-surface Cahn--Hilliard system:
\begin{equation}\label{CHCH}
\begin{alignedat}{3}
\pd_{t}u - \div (m(u) \nabla \lambda_{u}) &= 0, \quad \lambda_{u} = - \Lap u + W'(u) & \text{ in } Q, \\
\pd_{t}\phi - \div_{\Gamma} (n(\phi) \surf \lambda_{\phi}) + m(u) \pdnu \lambda_{u} & = 0, \quad  \lambda_{\phi} = -\LB \phi + W_{\Gamma}'(\phi) + h'(\phi) \pdnu u \quad & \text{ on } \Sigma, \\
K \pdnu u + u & = h(\phi), \quad  M m(u)\pdnu \lambda_{u} + \lambda_{u} = \lambda_{\phi} & \text{ on } \Sigma,
\end{alignedat}
\end{equation}
that satisfies the energy identity
\begin{align}\label{CHCH:Energy}
\frac{\dd}{\dt} E(u, \phi) + \int_{\Omega} m(u) \abs{\nabla \lambda_{u}}^{2} \dx + \int_{\Gamma} n(\phi) \abs{\surf \lambda_{\phi}}^{2} + M^{-1} \abs{\lambda_{u} - \lambda_{\phi}}^{2} \dHaus = 0.
\end{align}

\subsection{Bulk Allen--Cahn and surface Cahn--Hilliard system}  
Setting $\bm{J}_{u} = \bm{0}$ and choosing for a non-negative constant $\gamma$ and a non-negative mobility $n(\phi)$,
\begin{align*}
R = \lambda_{u}, \quad \bm{J}_{\phi} = - n(\phi) \surf \lambda_{\phi}, \quad Z = \gamma \, \lambda_{\phi},
\end{align*}
leads a bulk Allen--Cahn equation coupled to a surface Cahn--Hilliard-type equation:
\begin{equation}\label{ACCH}
\begin{alignedat}{3}
\pd_{t}u - \Lap u + W'(u) & = 0 & \text{ in } Q, \\
\pd_{t}\phi - \div_{\Gamma}(n(\phi) \surf \lambda_{\phi}) + \gamma \lambda_{\phi} & = 0, \quad \lambda_{\phi} = - \LB \phi + W_{\Gamma}'(\phi) + h'(\phi) \pdnu u & \quad \text{ on } \Sigma, \\
K \pdnu u + u & = h(\phi) & \text{ on } \Sigma,
\end{alignedat}
\end{equation}
that satisfies the energy identity
\begin{align}\label{ACCH:Energy}
\frac{\dd}{\dt} E(u, \phi) + \int_{\Omega} \abs{\pd_{t} u}^{2} \dx + \int_{\Gamma} n(\phi) \abs{\surf \lambda_{\phi}}^{2} + \gamma \abs{\lambda_{\phi}}^{2} \dHaus = 0.
\end{align}

\subsection{Bulk Cahn--Hilliard and surface Allen--Cahn system} 
Setting $\bm{J}_{\phi} = \bm{0}$ and choosing for a non-negative constant $\gamma$ and a non-negative mobility $m(u)$,
\begin{align*}
R = \gamma \, \lambda_{u}, \quad \bm{J}_{u} = - m(u) \nabla \lambda_{u}, \quad Z =  \lambda_{\phi},
\end{align*}
leads to a bulk Cahn--Hilliard-type equation coupled to a surface Allen--Cahn-type equation:
\begin{equation}\label{CHAC}
\begin{alignedat}{3}
\pd_{t}u - \div (m(u) \nabla \lambda_{u}) + \gamma \, \lambda_{u} & = 0, \quad \lambda_{u} = -\Lap u + W'(u) & \text{ in } Q, \\
\pd_{t}\phi & = \LB \phi - W_{\Gamma}'(\phi) - h'(\phi) \pdnu u & \quad \text{ on } \Sigma, \\
K \pdnu u + u & = h(\phi) & \text{ on } \Sigma,
\end{alignedat}
\end{equation}
that satisfies the energy identity
\begin{align}\label{CHAC:Energy}
\frac{\dd}{\dt} E(u, \phi) + \int_{\Omega} m(u) \abs{\nabla \lambda_{u}}^{2} + \gamma \abs{\lambda_{u}}^{2} \dx + \int_{\Gamma} \abs{\pd_{t}\phi}^{2} - m(u) \lambda_{u} \pdnu \lambda_{u} \dHaus = 0.
\end{align}
One can prescribe homogeneous Neumann conditions $m(u) \pdnu \lambda_{u} = 0$ or Robin conditions $m(u) \pdnu \lambda_{u} = a(g - \lambda_{u})$ for constant $ a> 0$ and given function $g$ to close the above system.

\begin{remark}[Limiting transmission conditions/Fast reaction limits]
Note that in the above cases, by formally sending $K, M \to 0$, we obtain the transmission conditions
\begin{align*}
u \vert_{\Sigma} = h(\phi), \quad \lambda_{u} \vert_{\Sigma} = \lambda_{\phi} \quad \text{ on } \Sigma.
\end{align*}
\end{remark}
\begin{remark}[Alternate derivation of equations]
To the authors' best knowledge, it appears that the above coupled bulk-surface systems involving Allen--Cahn-type or Cahn--Hilliard-type equations with transmission conditions such as $u \vert_{\Sigma} = \phi$ and $\lambda_{u} \vert_{\Sigma} = \lambda_{\phi}$ have not been derived from the viewpoint of mathematical modelling.  This motivates the current section to provide a derivation of these systems of equations from balance laws.  We are aware that the recent work of Liu and Wu  \cite{LiuWu} also provides a mathematical derivation of a coupled bulk-surface Cahn--Hilliard system (which can obtained as the limit $K \to 0$ of the system \eqref{CHCH} with $h(s) = s$, $m(u) = 1$, $n(\phi) = 1$ and replacing $M \pdnu \lambda_{u} + \lambda_{u} = \lambda_{\phi}$ with $\pdnu \lambda_{u} = 0$ on $\Sigma$ as a boundary condition).  This is done by means of an energetic variational approach that combines the least action principle and Onsager's principle of maximum energy dissipation.
\end{remark}

\section{Main results}\label{sec:main}
In this paper we focus on the Allen--Cahn/Allen--Cahn system \eqref{ACAC} with possibly non-smooth potentials $W$ and $W_{\Gamma}$.  By this we mean that $W = \hat \beta + \hat \pi$ (resp.~$W_{\Gamma} = \hat \beta_\Gamma + \hat \pi_\Gamma$) is a sum of a proper, convex and lower semicontinuous part $\hat \beta$ (resp.~$\hat \beta_\Gamma$) and a smooth non-convex part $\hat \pi$ (resp.~$\hat \pi_\Gamma$).  We recall that the subdifferential of $\hat \beta : \R \to [0,\infty]$, denoted by $\beta := \pd \hat \beta : \R \to 2^{\R}$, is a set-valued maximal monotone operator \cite{Barbu,Brezis,GL} defined as
\begin{align*}
\beta(x) = \pd \hat \beta(x) = \{ \xi \in \R \, : \, \hat \beta(y) - \hat \beta(x) \geq (\xi, y - x) \, \text{ for all } y \in \R \}.
\end{align*}
Furthermore, we introduce the effective domain of $\beta$, denoted by $D(\beta)$, as $D(\beta) := \{ r \in \R \, : \beta(r) \neq \emptyset \}$, which can be different from the whole real line, and denote by $\beta^{\circ}(x)$ the (unique) minimal element of the set $\beta(x)$ satisfying $\abs{\beta^{\circ}(x)} = \inf_{z \in \beta(x)} \abs{z}$.  In light of the possible non-smoothness of the potentials, the system \eqref{ACAC} should be expressed as

\begin{subequations}\label{ACAC:gen}
\begin{alignat}{3}
\pd_{t} u  = \Lap u - \xi - \pi(u) + f, \quad \xi \in \beta(u) & \text{ in } Q, \label{ACAC:bulk} \\
\pd_{t} \phi  = \LB \phi - \xi_\Gamma - \pi_{\Gamma}(\phi) - h'(\phi) \pdnu u + f_{\Gamma}, \quad \xi_\Gamma \in \beta_\Gamma(\phi) & \text{ on } \Sigma, \label{ACAC:surf} \\
K \pdnu u + u = h(\phi) & \text{ on } \Sigma, \label{ACAC:Robin} \\
u(0) = u_{0} \text{ in } \Omega, \quad \phi(0)  = \phi_{0} & \text{ on } \Gamma, \label{ACAC:ini}
\end{alignat}
\end{subequations}
where $f: Q \to \R$, $f_{\Gamma} : \Sigma \to \R$, $u_{0} : \Omega \to \R$, $\phi_{0} : \Gamma \to \R$ are given functions, while $\xi$ and $\xi_\Gamma$ are selections from the sets $\beta(u)$ and $\beta_\Gamma(\phi)$, respectively.

Let us state the assumptions:
\begin{enumerate}[label=$(\mathrm{A \arabic*})$, ref = $\mathrm{A \arabic*}$]
\item \label{ass:domain} $\Omega \subset \R^{3}$ is a bounded domain with smooth boundary $\Gamma$.
\item \label{ass:h} The function $h \in C^{2}(\R)$ satisfies $h', h'' \in L^{\infty}(\R)$.
\item \label{ass:beta} $\beta$ and $\beta_{\Gamma}$ are maximal monotone graphs on $\R \times \R$ with effective domains $D(\beta)$ and $D(\beta_{\Gamma})$, respectively, and are the subdifferentials $\beta = \pd \hat{\beta}$, $\beta_{\Gamma} = \pd \hat{\beta}_{\Gamma}$ of some proper, lower semicontinuous and convex functions $\hat{\beta}, \hat{\beta}_{\Gamma} : \R \to [0,\infty]$ with $\hat{\beta}(0) = 0$, $\hat{\beta}_{\Gamma}(0) = 0$.  Furthermore, for all $\delta > 0$ there exists $C_{\delta} > 0$ such that
\begin{align*}
\abs{\xi} \leq \delta \, u \, \xi + C_{\delta} \text{ for all } \xi \in \beta(u).
\end{align*}
\item \label{ass:pi} $\pi, \pi_{\Gamma} : \R \to \R$ are Lipschitz continuous (with Lipschitz constants $L_{\pi}$ and $L_{\pi_{\Gamma}}$, respectively) and  their anti-derivatives satisfy $\hat{\pi}(s), \hat{\pi}_{\Gamma}(s) \geq 0$ for all $s \in \R$.
\item \label{ass:f} $f \in H^{1}(0,T; L^{2}(\Omega))$, $f_{\Gamma} \in H^{1}(0,T; L^{2}(\Gamma))$ for any $T \in (0,\infty)$.
\item  \label{ass:ini} The initial data satisfy $u_{0} \in H^{2}(\Omega)$ with $\beta^{\circ}(u_{0}) \in L^{2}(\Omega)$ and $\phi_{0} \in H^{2}(\Gamma)$ with $\beta_{\Gamma}^{\circ}(\phi_{0}) \in L^{2}(\Gamma)$.  In addition, the compatibility condition $K \pdnu u_{0} + u_{0} = h(\phi_{0})$ holds on $\Gamma$.
\end{enumerate}

\begin{remark}\label{rem:initialdata}
Note that $\beta$ and $\beta_{\Gamma}$ induce maximal monotone operators on $L^{2}(\Omega)$ and $L^{2}(\Gamma)$, respectively.  These operators are characterized by the pointwise inclusion in the following sense:
\begin{align*}
\xi \in L^{2}(\Omega) \text{ with } \xi \in \beta(u) \text{ for } u \in L^{2}(\Omega) \Longleftrightarrow \xi(x) \in \beta(u(x)) \text{ for a.e. } x \in \Omega.
\end{align*}
Furthermore, as $L^{2}(\Omega)$ is a Hilbert space, there exists a (unique) minimal element $\beta^{\circ}(u_{0})$, due to projections in Hilbert spaces, that satisfies $\norm{\beta^{\circ}(u_{0})}_{L^{2}(\Omega)} := \inf_{z \in \beta(u_{0})} \norm{z}_{L^{2}(\Omega)}$, and such minimal elements always exist.   Let us also point out that by the definition of subdifferential, the assertions $\beta^{\circ}(u_{0}) \in L^{2}(\Omega)$ and $\hat{\beta}(0) = 0$ immediately imply $\hat{\beta}(u_{0}) \in L^{1}(\Omega)$.  This is thanks to the relation
\begin{align*}
0 \leq \int_{\Omega} \hat{\beta}(u_{0}) \dx \leq \int_{\Omega} \beta^{\circ}(u_{0}) u_{0} \, dx < \infty.
\end{align*}
Analogous assertions also hold for $\beta_{\Gamma}^{\circ}(\phi_{0})$ and $\hat{\beta}_{\Gamma}(\phi_{0})$.  
\end{remark}

\subsection{Strong well-posedness}

Our first result concerns the strong existence of solutions to \eqref{ACAC:gen}. 
\begin{thm}[Strong existence]\label{thm:Exist}
For any $T > 0$, under assumptions \eqref{ass:domain}-\eqref{ass:ini} there exists a quadruple $(u,\phi, \xi, \xi_{\Gamma})$ with
\begin{align*}
u & \in L^{\infty}(0,T;H^{2}(\Omega)), \quad \pdnu u \in H^{1}(0,T;L^{2}(\Gamma)), \\
\pd_{t} u & \in L^{\infty}(0,T;L^{2}(\Omega)) \cap L^{2}(0,T;H^{1}(\Omega)), \quad \pd_{t}u \vert_{\Sigma} \in L^{2}(0,T;L^{2}(\Gamma)), \\
\xi & \in L^{\infty}(0,T;L^{2}(\Omega)), \quad \xi \in \beta(u) \text{ a.e. in } \Omega,\\
\phi & \in L^{\infty}(0,T;H^{2}(\Gamma)), \\
\pd_{t} \phi & \in L^{\infty}(0,T;L^{2}(\Gamma)) \cap L^{2}(0,T;H^{1}(\Gamma)), \\
\xi_{\Gamma} & \in L^{\infty}(0,T;L^{2}(\Gamma)), \quad \xi_{\Gamma} \in \beta_{\Gamma}(\phi) \text{ a.e. on } \Sigma,
\end{align*}
and satisfies \eqref{ACAC:bulk} a.e in $Q$ and \eqref{ACAC:surf}, \eqref{ACAC:Robin} a.e.~on $\Sigma$ and also \eqref{ACAC:ini}.
\end{thm}

We mention that due to the the embedding $L^{\infty}(0,T;H^{2}(\Omega)) \cap H^{1}(0,T;L^{2}(\Omega)) \subset C^{0}([0,T];H^{r}(\Omega))$ for any $r < 2$ and the trace theorem, the normal derivative $\pdnu u$ on $\Gamma$ is continuous up to initial time, and thus the initial condition $u_{0}$ has to satisfy the compatibility condition outlined in \eqref{ass:ini}.

Our second result is the continuous dependence on the data.
\begin{thm}[Continuous dependence]\label{thm:Ctsdep}
Under assumptions \eqref{ass:domain}-\eqref{ass:ini}, let $\{(u_{i}, \phi_{i})\}_{i=1,2}$ denote strong solutions to \eqref{ACAC:gen} corresponding to data $\{(u_{0,i}, \phi_{0,i}, f_{i}, f_{\Gamma,i})\}_{i=1,2}$.  Then, there exists a positive constant $C$, depending on $\Gamma$, $\norm{\pdnu u_{i}}_{L^{\infty}(0,T;L^{2}(\Gamma))}$, the Lipschitz constants of $\pi$, $\pi_{\Gamma}$, $h'$ and $h''$, the fixed time $T$, and $K$, such that
\begin{equation}
\begin{aligned}
& \norm{u_{1} - u_{2}}_{L^{\infty}(0,T;L^{2}(\Omega)) \cap L^{2}(0,T;H^{1}(\Omega))} + \norm{\phi_{1} - \phi_{2}}_{L^{\infty}(0,T;L^{2}(\Gamma)) \cap L^{2}(0,T;H^{1}(\Gamma))} \\
& \quad \leq C \left ( \norm{u_{0,1} - u_{0,2}}_{L^{2}(\Omega)} + \norm{\phi_{0,1} - \phi_{0,2}}_{L^{2}(\Gamma)} + \norm{f_{1} - f_{2}}_{L^{2}(Q)} + \norm{f_{\Gamma,1} - f_{\Gamma,2}}_{L^{2}(\Sigma)} \right ).
\end{aligned}
\end{equation}
\end{thm}

As a consequence of Theorem \ref{thm:Ctsdep}, the strong solution obtained from Theorem \ref{thm:Exist} is unique.

\subsection{Omega-limit set}

The first and second theorems show that there exists a unique strong solution on any finite time interval, which allows us to address the long-time behavior of \eqref{ACAC:gen}.  Our third result deals with the omega-limit set of an arbitrary initial datum $(u_{0}, \phi_{0})$ satisfying \eqref{ass:ini}.  We make the following additional assumptions.

\begin{enumerate}[resume*]
\item \label{ass:long:W} There exist positive constants $c_{1}$ and $c_{2}$ such that
\begin{align*}
\hat{\beta}(s) + \hat{\pi}(s) \geq c_{1} \abs{s}^{2} - c_{2}, \quad \hat{\beta}_{\Gamma}(s) + \hat{\pi}_{\Gamma}(s) \geq c_{1} \abs{s}^{2} - c_{2}.
\end{align*}
\item \label{ass:long:f} In addition to \eqref{ass:f}, the functions $f$ and $f_{\Gamma}$ satisfy
\begin{align*}
f  \in H^{1}(0,\infty;L^{2}(\Omega)), \quad f_{\Gamma} \in H^{1}(0,\infty;L^{2}(\Gamma)).
\end{align*}
\end{enumerate}
We point out that \eqref{ass:long:f} implies $f(t) \to 0$ in $L^{2}(\Omega)$, $f_{\Gamma}(t) \to 0$ in $L^{2}(\Gamma)$ as $t \to \infty$ by virtue of belonging to the Bochner space $H^{1}(0,\infty;X)$ where $X = L^{2}(\Omega)$ or $X = L^{2}(\Gamma)$.

\begin{thm}[Omega-limit set]\label{thm:Lt}
Under assumptions \eqref{ass:domain}-\eqref{ass:long:f}, the omega-limit set
\begin{align*}
\omega := \Big{\{} (u_{\infty}, \phi_{\infty}) & \, : \, \exists \{t_{n}\}_{n \in \N}, \, t_{n} \nearrow \infty \\
& \text{ and } (u(t_{n}), \phi(t_{n})) \to (u_{\infty}, \phi_{\infty}) \text{ weakly in } H^{1}(\Omega) \times H^{1}(\Gamma) \Big{\}}
\end{align*}
is non-empty.  Moreover, if $(u_{\infty}, \phi_{\infty})$ is an element of $\omega$, then $u_{\infty} \in H^{2}(\Omega)$ and $\phi_{\infty} \in H^{2}(\Gamma)$ satisfy
\begin{subequations}\label{stationary}
\begin{alignat}{3}
\Lap u_{\infty} - \pi(u_{\infty}) \in \beta(u_{\infty}) & \text{ a.e. in } \Omega, \\
\LB \phi_{\infty} - \pi_{\Gamma}(\phi_{\infty}) - h'(\phi_{\infty}) \pdnu u_{\infty}  \in \beta_{\Gamma}(\phi_{\infty}) & \text{ a.e. on } \Gamma, \\
K \pdnu u_{\infty} + u_{\infty} = h(\phi_{\infty}) & \text{ a.e. on } \Gamma.
\end{alignat}
\end{subequations}
\end{thm}

\subsection{Weak well-posedness to limit problem}

A natural question is whether the solutions to \eqref{ACAC:gen} converge in the limit as $K \to 0$.  We expect that in the limit the transmission condition
\begin{align*}
u \vert_{\Sigma} = h(\phi)  \text{ on } \Sigma
\end{align*}
will be obtained.  Our fourth result provides a positive answer for the case where $h$ is an affine linear function.

\begin{thm}[Weak solutions to limit problem]\label{thm:limitK}
In addition to assumptions \eqref{ass:domain}-\eqref{ass:ini}, suppose further that
\begin{enumerate}[resume*, leftmargin=*]
\item \label{Kto0:h} $h$ is affine linear, i.e., $h(s) = \alpha s + \eta$ for $\alpha \neq 0$, $\eta \in \R$.
\item \label{Kto0:beta} $\beta = \hat{\beta}', \beta_{\Gamma} = \hat{\beta}_{\Gamma}'$ are continuous, monotone, single-valued functions, and there exist positive constants $C_{1}, \dots, C_{4}$, such that for some $p \leq 5$, $q < \infty$ and for all $s \in \R$,
\begin{align*}
\abs{\beta(s)} \leq C_{1} \abs{s}^{p} + C_{2}, \quad \abs{\beta_{\Gamma}(s)} \leq C_{3} \abs{s}^{q} + C_{4}.
\end{align*}
\end{enumerate}
For each $K > 0$, let $(u_{K}, \phi_{K})$ denote a strong solution to \eqref{ACAC:gen} with data $(u_{0,K}, \phi_{0,K}, f_{K}, f_{\Gamma,K})$ satisfying \eqref{ass:f}, \eqref{ass:ini} and
\begin{equation}\label{Klim:ass}
\begin{aligned}
f_{K} \to f \text{ in } L^{2}(Q), & \quad f_{\Gamma,K} \to f_{\Gamma}  \text{ in } L^{2}(\Sigma), \\
u_{0,K} \to u_{0}  \text{ in } H^{1}(\Omega), & \quad \phi_{0,K} \to \phi_{0}  \text{ in } L^{2}(\Gamma) \\
\qquad \text{ such that } u_{0} = h(\phi_{0}) & \text{ and } \norm{u_{0,K} - h(\phi_{0,K})}_{L^{2}(\Gamma)}^{2} \leq C K,
\end{aligned}
\end{equation}
where $C$ is a positive constant independent of $K$.  Then, for any $T \in (0,\infty)$, there exists a pair of functions $(u,\phi)$ such that
\begin{align*}
u_{K} & \to u \text{ weakly-* in } L^{\infty}(0,T;H^{1}(\Omega)) \cap H^{1}(0,T;L^{2}(\Omega)), \\
\phi_{K} & \to \phi \text{ weakly-* in } L^{\infty}(0,T;H^{1}(\Gamma)) \cap H^{1}(0,T;L^{2}(\Gamma)), \\
u_{K} - h(\phi_{K}) & \to 0 \text{ strongly in } L^{2}(0,T;L^{2}(\Gamma)),
\end{align*}
and satisfies for all $\zeta \in H^{1}(\Omega)$ such that $\zeta_{\Gamma} := \zeta \vert_{\Gamma} \in H^{1}(\Gamma)$, and for a.e. $t \in (0,T)$,
\begin{equation}\label{Kto0:weakform}
\begin{aligned}
0 & = \int_{\Omega} \pd_{t} u(t) \, \zeta + \nabla u(t) \cdot \nabla \zeta + \beta(u(t)) \, \zeta + \pi(u(t)) \, \zeta - f(t) \, \zeta \dx \\
& \quad + \int_{\Gamma} \frac{1}{\alpha} \left ( \pd_{t}\phi(t) \, \zeta_{\Gamma} + \surf \phi(t) \cdot \surf \zeta_{\Gamma} + \beta_{\Gamma}(\phi(t)) \, \zeta_{\Gamma} + \pi_{\Gamma}(\phi(t)) \, \zeta_{\Gamma} - f_{\Gamma}(t) \, \zeta_{\Gamma} \right ) \dHaus .
\end{aligned}
\end{equation}
Furthermore, it holds that
\begin{align*}
& \beta(u) \in L^{\infty}(0,T;L^{\frac{6}{p}}(\Omega)), \quad \beta_{\Gamma}(\phi) \in L^{\infty}(0,T;L^{s}(\Gamma)), \\
& u \vert_{\Sigma} = h(\phi) = \alpha \phi + \eta \quad \text{ a.e. on } \Sigma,
\end{align*}
for any $s \in [1,\infty)$ and $p$ is the exponent in \eqref{Kto0:beta}.  

\bigskip

For maximal monotone graphs $\beta$ and $\beta_{\Gamma}$, an analogous result also holds if we replace assumption \eqref{Kto0:beta} with
\begin{enumerate}[resume*]
\item \label{Kto0:maxmono} For some $p < 5$ and $q < \infty$, there exist positive constants $C_{1}, \dots, C_{4}$ such that
\begin{align*}
\abs{\xi} \leq C_{1} \abs{u}^{p} + C_{2}, \quad \abs{\xi_{\Gamma}} \leq C_{3} \abs{\phi}^{q} + C_{4}
\end{align*}
for all $\xi \in \beta(u)$ and $\xi_{\Gamma} \in \beta_{\Gamma}(\phi)$.
\end{enumerate}
Then, in \eqref{Kto0:weakform} we replace $\beta(u)$ with $\xi$ and $\beta_{\Gamma}(\phi)$ with $\xi_{\Gamma}$.
\end{thm}

\begin{remark}
Equation \eqref{Kto0:weakform} is the variational formulation of the limit problem
\begin{subequations}\label{Limit:form:h}
\begin{alignat}{3}
\pd_{t} u = \Lap u - \beta(u) - \pi(u) + f & \text{ in } Q, \\
\pd_{t}\phi = \LB \phi - \beta_{\Gamma}(\phi) - \pi_{\Gamma}(\phi) - h'(\phi) \pdnu u + f_{\Gamma} & \text{ on } \Sigma, \\
u = h(\phi) = \alpha \phi + \eta & \text{ on } \Sigma.
\end{alignat}
\end{subequations}
In the special case $\alpha = 1$ and $\eta = 0$, we have $\phi = u \vert_{\Sigma}$ and this is the setting of Calatroni and Colli \cite{CalaColli} and also Colli and Fukao \cite{ColliFukaoAC}.  While the assumption \eqref{Kto0:maxmono} for maximal monotone graphs $\beta$ and $\beta_{\Gamma}$ is not as general as assumed in \cite{CalaColli,ColliFukaoAC}, here we do not require a compatibility condition, cf. \cite[(2.22)-(2.23)]{CalaColli}, since we directly obtain estimates for the selections $\xi$ and $\xi_{\Gamma}$ with the growth conditions \eqref{Kto0:maxmono} and the estimates for $u$ and $\phi$.
\end{remark}

\begin{thm}[Continuous dependence]\label{thm:Limit:ctsdep}
Let $\{(u_{i}, \phi_{i})\}_{i=1,2}$ denote two weak solutions to \eqref{Limit:form:h} with data $\{(u_{0,i}, \phi_{0,i}, f_{i}, f_{\Gamma,i})\}_{i=1,2}$ satisfying
\begin{align*}
u_{i} \vert_{\Sigma} = h(\phi_{i}) = \alpha \phi_{i} + \eta, \quad u_{0,i} \vert_{\Gamma} = h(\phi_{0,i}).
\end{align*}
Then, there exists a positive constant $C$, depending only on $T$ and the Lipschitz constants of $\pi$ and $\pi_{\Gamma}$ such that
\begin{align*}
& \norm{u_{1} - u_{2}}_{L^{\infty}(0,T;L^{2}(\Omega)) \cap L^{2}(0,T;H^{1}(\Omega))} + \norm{\phi_{1} - \phi_{2}}_{L^{\infty}(0,T;L^{2}(\Gamma)) \cap L^{2}(0,T;H^{1}(\Gamma))} \\
& \quad \leq C \left ( \norm{u_{0,1} - u_{0,2}}_{L^{2}(\Omega)} + \norm{\phi_{0,1} - \phi_{0,2}}_{L^{2}(\Gamma)} + \norm{f_{1} - f_{2}}_{L^{2}(Q)} + \norm{f_{\Gamma,1} - f_{\Gamma,2}}_{L^{2}(\Sigma)} \right ).
\end{align*}
\end{thm}

As a consequence, any weak solution to \eqref{Limit:form:h} is unique.

\subsection{Existence of strong solutions to the limit problem}
For a convex, proper, lower semicontinuous function $f: \R \to [0,\infty]$ and an affine linear function $g: \R \to \R$, $g(s) = \alpha^{-1}(s-\eta)$ where $\alpha \neq 0$ and $\eta \in \R$, it is clear from the definition of the subdifferential that
\begin{align*}
\pd (f \circ g)(z) = \frac{1}{\alpha} \pd f(g(z)) 
\end{align*}
for all $z \in D(f \circ g) := \{ y \in \R : f(g(y)) < \infty \}$.  Our aim is to express \eqref{Intro:ACAC:lim:g} (for affine linear $g$) as an abstract equation, and appeal to the framework of Colli and Fukao \cite{ColliFukaoAC} to prove strong existence, while continuous dependence on data and uniqueness follow readily from Theorem \ref{thm:Limit:ctsdep}.  To achieve this let us introduce the following Hilbert spaces
\begin{align}\label{Prod:Space}
\bm{H} = L^{2}(\Omega) \times L^{2}(\Gamma), \quad \bm{V} = \{(a,b) \in H^{1}(\Omega) \times H^{1}(\Gamma) : a \vert_{\Gamma} = b \},
\end{align}
with the inner product
\begin{align}\label{Prod:Space:innerProd}
(\bm{p}, \bm{q})_{\bm{X}} = (p,q)_{X} + (p_{\Gamma}, q_{\Gamma})_{X_{\Gamma}}
\end{align}
for $\bm{p} = (p,p_{\Gamma})$, $\bm{q} = (q, q_{\Gamma})$, $\bm{X} = X \times X_{\Gamma}$ for $X = L^{2}(\Omega)$ or $H^{1}(\Omega)$, and for $X_{\Gamma} = L^{2}(\Gamma)$ or $H^{1}(\Gamma)$.  We also introduce the proper, lower semicontinuous and convex functional $\varphi: \bm{H} \to [0,\infty]$ by
\begin{align}
\label{CF:varphi}
\varphi(\bm{z}) = \begin{cases}
\displaystyle \int_{\Omega} \frac{1}{2} \abs{\nabla z}^{2} + \hat{\beta}(z) \dx + \int_{\Gamma} \frac{1}{2 \alpha^{2}} \abs{\surf z_{\Gamma}}^{2} + \hat{\beta}_{\Gamma}(g(z_{\Gamma})) \dHaus \\[2ex]
\quad \quad  \text{ if } \bm{z} = (z, z_{\Gamma}) \in \bm{V}, \quad \hat{\beta}(z) \in L^{1}(\Omega), \; \hat{\beta}_{\Gamma}(g(z_{\Gamma})) \in L^{1}(\Gamma), \\
+ \infty \quad \text{ otherwise},
\end{cases}
\end{align}
whose subdifferential $\pd \varphi$ can be characterized formally as
\begin{align*}
& \bm{y} = (y, y_{\Gamma}) \in \pd \varphi(\bm{z}) \text{ is an element of } \bm{H} \text{ if and only if } \\
& \quad (y, y_{\Gamma}) = (-\Lap z + \beta(z), -\alpha^{-2} \LB z_{\Gamma} + \alpha^{-1} \beta_{\Gamma}(g(z_{\Gamma})) + \pdnu z).
\end{align*} 
Then, we can write \eqref{Intro:ACAC:lim:g} into a single abstract equation for $\bm{u} := (u, u_{\Gamma})$:
\begin{equation}\label{Abs:Equ}
\begin{aligned}
\AA^{2} \bm{u}'(t) + \pd \varphi(\bm{u}(t)) + \AA \left ( \bm{\pi}(\bm{u}(t)) - \bm{f}(t) \right ) \ni \bm{0} & \text{ in } \bm{H} \text{ for a.e. } t \in (0,T), \\
\bm{u}(0) = \bm{u}_{0} & \text{ in } \bm{H},
\end{aligned}
\end{equation}
where $\bm{u}_{0} = (u_{0}, u_{0} \vert_{\Gamma})$, $\bm{\pi}(\bm{u}) := (\pi(u), \pi_{\Gamma}(g(u_{\Gamma})))$, $\bm{f} := (f, f_{\Gamma})$, and $\AA$ is the constant matrix
\begin{align}\label{mat}
\AA = \begin{pmatrix}
1 & 0 \\ 0 & \alpha^{-1} \end{pmatrix}.
\end{align}

Our fifth result concerning the strong existence to the limit problem is formulated as follows.

\begin{thm}[Strong existence of the limit problem]\label{thm:Limit:Exist}
For any $T > 0$, $\eta \in \R$ and $\alpha \neq 0$, under assumptions \eqref{ass:domain}-\eqref{ass:f} and
\begin{enumerate}[resume*]
\item \label{Limit:strong:beta} For any $q,r < \infty$, there exists a positive constant $C$ such that
\begin{align*}
\abs{\xi} \leq C \left ( 1 + \abs{u}^{q} \right ), \quad \abs{\xi_{\Gamma}} \leq C \left ( 1 + \abs{u_{\Gamma}}^{r} \right )
\end{align*}
for all $\xi \in \beta(u)$ and $\xi_{\Gamma} \in \beta_{\Gamma}(u_{\Gamma})$. 
\item \label{ass:ini:strong} The initial data satisfy $u_{0} \in H^{2}(\Omega)$ with $\beta^{\circ}(u_{0}) \in L^{2}(\Omega)$ and trace $u_{0} \vert_{\Gamma} \in H^{2}(\Gamma)$ with $\beta_{\Gamma}^{\circ}(\alpha^{-1}(u_{0} \vert_{\Gamma} - \eta)) \in L^{2}(\Gamma)$.
\end{enumerate}
Then, there exists a unique strong solution $(u, u_{\Gamma}, \xi, \xi_{\Gamma})$ satisfying
\begin{align*}
u & \in L^{\infty}(0,T;H^{2}(\Omega)) \cap H^{1}(0,T;H^{1}(\Omega)) \cap W^{1,\infty}(0,T;L^{2}(\Omega)), \\
u_{\Gamma} & \in L^{\infty}(0,T;H^{2}(\Gamma)) \cap H^{1}(0,T;H^{1}(\Gamma)) \cap W^{1,\infty}(0,T;L^{2}(\Gamma)), \\
\xi & \in L^{\infty}(0,T;L^{2}(\Omega)) \text{ with } \xi \in \beta(u) \text{ a.e. in } Q, \\
\xi_{\Gamma} & \in L^{\infty}(0,T;L^{2}(\Gamma)) \text{ with } \xi_{\Gamma} \in \beta_{\Gamma}(\alpha^{-1}(u_{\Gamma} - \eta)) \text{ a.e. on } \Sigma, 
\end{align*}
and
\begin{subequations}\label{ACAC:limit}
\begin{alignat}{3}
\pd_{t} u = \Lap u - \xi - \pi(u) + f & \text{ in } Q, \\
\pd_{t} u_{\Gamma} = \LB u_{\Gamma} - \alpha \left ( \xi_{\Gamma} + \pi_{\Gamma}(\alpha^{-1}(u_{\Gamma} - \eta)) - f_{\Gamma} \right ) - \alpha^{2} \pdnu u & \text{ on } \Sigma, \\
u_{\Gamma} = u \vert_{\Sigma} & \text{ on } \Sigma, \\
u(0) = u_{0} \text{ in } \Omega, \quad u_{\Gamma}(0) = u_{0} \vert_{\Gamma} & \text{ on } \Gamma.
\end{alignat}
\end{subequations}
\end{thm}
Let us mention that in \eqref{Limit:strong:beta} we allow arbitrary polynomial growth for the maximal monotone graph $\beta$, which is in contrast to the exponent as assumed in \eqref{Kto0:maxmono}.

\subsection{Error estimate}

Due to the existence of strong solutions for the limit problem \eqref{Limit:form:h}, we can also derive a convergence rate.  This is detailed in our sixth result below.  

\begin{thm}[Convergence rate]\label{thm:rates}
Let $\mathbb{X}_{\Omega} := L^{\infty}(0,T;L^{2}(\Omega)) \cap L^{2}(0,T;H^{1}(\Omega))$ and $\mathbb{X}_{\Gamma} := L^{\infty}(0,T;L^{2}(\Gamma)) \cap L^{2}(0,T;H^{1}(\Gamma))$.  Under assumption \eqref{Kto0:h}, \eqref{Kto0:beta} $($or \eqref{Kto0:maxmono}$)$, for $K > 0$, let $(u_{K}, \phi_{K})$ denote the unique strong solution to \eqref{ACAC:gen} with data $(u_{0,K}, \phi_{0,K}, f_{K}, f_{\Gamma,K})$ obtained from Theorems \ref{thm:Exist} and \ref{thm:Ctsdep}, and let $(u,\phi)$ denote the unique strong solution to \eqref{Limit:form:h} with data $(u_{0}, \phi_{0}, f, f_{\Gamma})$ obtained from Theorem \ref{thm:Limit:Exist}, where we set $\phi := \alpha^{-1}(u_{\Gamma} - \eta)$ and $\phi_{0} := \alpha^{-1} (u_{0} \vert_{\Gamma} - \eta)$.  Then, there exists a positive constant $C$ depending only on the Lipschitz constants of $\pi$, $\pi_{\Gamma}$ and on $T$, such that
\begin{equation}\label{Rate:Est:1}
\begin{aligned}
& \norm{u_{K} - u}_{\mathbb{X}_{\Omega}}^{2} + \norm{\phi_{K} - \phi}_{\mathbb{X}_{\Gamma}}^{2} + K^{-1} \norm{\alpha \phi_{K} + \eta - u_{K}}_{L^{2}(\Sigma)}^{2} \\
& \quad \leq C \left ( \norm{f_{K} - f}_{L^{2}(Q)}^{2} + \norm{f_{\Gamma,K} - f_{\Gamma}}_{L^{2}(\Sigma)}^{2} + \norm{u_{0,K} - u_{0}}_{L^{2}(\Omega)}^{2} + \norm{\phi_{0,K} - \phi_{0}}_{L^{2}(\Gamma)}^{2} \right ) \\
& \qquad + C K \norm{\pdnu u}_{L^{2}(0,T;L^{2}(\Gamma))}^{2}.
\end{aligned}
\end{equation}
Furthermore, in addition to the assumption \eqref{Klim:ass} for the data $(u_{0,K}, \phi_{0,K}, f_{K}, f_{\Gamma,K})$, suppose that there exist positive constants $C$, $\theta_{1}$, $\theta_{2}$, $\theta_{3}$ and $\theta_{4}$, not depending on $K$, such that
\begin{align*}
\norm{f_{K} - f}_{L^{2}(Q)} & \leq C K^{\theta_{1}}, \quad \norm{f_{\Gamma,K} - f_{\Gamma}}_{L^{2}(\Sigma)} \leq C K^{\theta_{2}}, \\
\norm{u_{0,K} - u_{0}}_{L^{2}(\Omega)} & \leq C K^{\theta_{3}}, \quad \norm{\phi_{0,K} - \phi_{0}}_{L^{2}(\Gamma)} \leq C K^{\theta_{4}}.
\end{align*}
Then, it holds that
\begin{equation}\label{Rate:Est:2}
\begin{aligned}
& \norm{u_{K} - u}_{\mathbb{X}_{\Omega}} + \norm{\phi_{K} - \phi}_{\mathbb{X}_{\Gamma}} + K^{-\frac{1}{2}} \norm{\alpha \phi_{K} + \eta - u_{K}}_{L^{2}(\Sigma)} \leq C K^{\theta}
\end{aligned}
\end{equation}
for $\theta := \min( \frac{1}{2}, \theta_{1}, \theta_{2}, \theta_{3}, \theta_{4})$.
\end{thm}

\section{Continuous dependence}\label{sec:ctsdep}
Let $\{(u_{0,i}, \phi_{0,i}, f_{i}, f_{\Gamma,i})\}_{i=1,2}$ denote two sets of data and $\{(u_{i}, \phi_{i})\}_{i=1,2}$ the corresponding solutions to \eqref{ACAC:gen} with differences denoted by $\hat{u}$, $\hat{\phi}$, $\hat{u}_{0}$, $\hat{\phi}_{0}$, $\hat{f}$ and $\hat{f}_{\Gamma}$, respectively.  Then, testing the difference of \eqref{ACAC:bulk} with $\hat{u}$, the difference of \eqref{ACAC:surf} with $\hat{\phi}$ and adding, using the monotonicity of $\beta$ and $\beta_{\Gamma}$, we arrive at
\begin{align*}
& \frac{1}{2} \frac{\dd}{\dt} \left ( \norm{\hat{u}}_{L^{2}(\Omega)}^{2} + \norm{\hat{\phi}}_{L^{2}(\Gamma)}^{2} \right ) + \norm{\nabla \hat{u}}_{L^{2}(\Omega)}^{2} + \norm{\surf \hat{\phi}}_{L^{2}(\Gamma)}^{2} \\
& \quad \leq \norm{\hat{f}}_{L^{2}(\Omega)} \norm{\hat{u}}_{L^{2}(\Omega)} + L_{\pi} \norm{\hat{u}}_{L^{2}(\Omega)}^{2} + \norm{\hat{f}_{\Gamma}}_{L^{2}(\Gamma)} \norm{\hat{\phi}}_{L^{2}(\Gamma)} + L_{\pi_{\Gamma}} \norm{\hat{\phi}}_{L^{2}(\Gamma)}^{2} \\
& \qquad + \int_{\Gamma} \pdnu \hat{u} \, \hat{u} - h'(\phi_{1}) \pdnu \hat{u} \, \hat{\phi} - \pdnu u_{2} (h'(\phi_{1}) - h'(\phi_{2})) \hat{\phi} \dHaus,
\end{align*}
where $L_{\pi}$ and $L_{\pi_{\Gamma}}$ are the Lipschitz constants of $\pi$ and $\pi_{\Gamma}$.  A close inspection of the last term on the right-hand side yields that
\begin{align*}
& \int_{\Gamma} \pdnu \hat{u} \, \hat{u} - h'(\phi_{1}) \pdnu \hat{u} \, \hat{\phi} - \pdnu u_{2} (h'(\phi_{1}) - h'(\phi_{2})) \hat{\phi} \dHaus \\
& \quad = \int_{\Gamma} K^{-1} (h(\phi_{1})-h(\phi_{2}) - \hat{u}) (\hat{u} - h'(\phi_{1})\hat{\phi}) - \pdnu u_{2}(h'(\phi_{1})-h'(\phi_{2})) \hat{\phi} \dHaus \\
 & \quad \leq - K^{-1} \norm{\hat{u}}_{L^{2}(\Gamma)}^{2} + K^{-1} L_{h} \norm{\hat{\phi}}_{L^{2}(\Gamma)} \norm{\hat{u}}_{L^{2}(\Gamma)} + C K^{-1} \left ( L_{h} \norm{\hat{\phi}}_{L^{2}(\Gamma)} + \norm{\hat{u}}_{L^{2}(\Gamma)} \right ) \norm{\hat{\phi}}_{L^{2}(\Gamma)} \\
 & \qquad + \norm{\pdnu u_{2}}_{L^{2}(\Gamma)} L_{h'} \norm{\hat{\phi}}_{L^{4}(\Gamma)}^{2},
\end{align*}
where $L_{h'}$ is the Lipschitz constant of $h'$, and $C$ is a positive constant depending on $\norm{h'}_{L^{\infty}(\R)}$.  Since $\pdnu u_{2} \in H^{1}(0,T;L^{2}(\Gamma)) \subset L^{\infty}(0,T;L^{2}(\Gamma))$, by the interpolation estimate
\begin{align*}
\norm{\hat{\phi}}_{L^{4}(\Gamma)}^{2} \leq \eps \norm{\surf \hat{\phi}}_{L^{2}(\Gamma)}^{2} + C_{\eps} \norm{\hat{\phi}}_{L^{2}(\Gamma)}^{2},
\end{align*}
and Gronwall's inequality we obtain
\begin{align*}
& \norm{\hat{u}}_{L^{\infty}(0,T;L^{2}(\Omega)) \cap L^{2}(0,T;H^{1}(\Omega))} + \norm{\hat{\phi}}_{L^{\infty}(0,T;L^{2}(\Gamma)) \cap L^{2}(0,T;H^{1}(\Gamma))} \\
& \quad \leq C \left ( \norm{\hat{u}_{0}}_{L^{2}(\Omega)} + \norm{\hat{\phi}_{0}}_{L^{2}(\Gamma)} + \norm{\hat{f}}_{L^{2}(0,T;L^{2}(\Omega))} + \norm{\hat{f}_{\Gamma}}_{L^{2}(0,T;L^{2}(\Gamma))} \right ).
\end{align*}

\section{Existence}\label{sec:exist}
In this section, we introduce a two-level approximation and provide a number of a priori estimates.

\subsection{Approximation scheme}
For $\eps \in (0,1)$, we recall that the Yosida approximation $\hat{\beta}_{\eps}$ of the proper, lower semicontinuous function $\hat{\beta} : \R \to [0,\infty]$ is defined as
\begin{align*}
\hat{\beta}_{\eps}(s) := \inf_{y \in \R} \left ( \hat{\beta}(y) + \frac{1}{2 \eps} \abs{s - y}^{2} \right ),
\end{align*}
which satisfies
\begin{align}\label{Yosida}
0 \leq \hat{\beta}_{\eps}(s) \leq \hat{\beta}(s) \quad \forall s \in \R \quad \text{ and } \quad \hat{\beta}_{\eps}(s) \nearrow \hat{\beta}(s) \quad \text{ as } \eps \searrow 0.
\end{align}
For a short introduction to maximal monotone operators and the Yosida approximation we refer the reader to \cite{Barbu,Brezis,GL} and the references cited therein.  We replace $\beta(u)$ and $\beta_{\Gamma}(\phi)$ in \eqref{ACAC:gen} with $\beta_{\eps}(u)$ and $\beta_{\Gamma,\eps}(\phi)$, respectively, and use a Galerkin procedure to establish the existence of solutions.  Let us point out that by \eqref{ass:ini}, \eqref{Yosida} and Remark \ref{rem:initialdata}, there exists a positive constant $C$, independent of $\eps$ such that
\begin{align}\label{Yosida:ini}
\norm{\hat{\beta}_{\eps}(u_{0})}_{L^{1}(\Omega)} \leq \norm{\hat{\beta}(u_{0})}_{L^{1}(\Omega)} \leq C, \quad \norm{\hat{\beta}_{\Gamma, \eps}(\phi_{0})}_{L^{1}(\Gamma)} \leq \norm{\hat{\beta}_{\Gamma}(\phi_{0})}_{L^{1}(\Gamma)} \leq C.
\end{align}
Let $\{y_{j}\}_{j \in \N}$ be a basis that is orthonormal in $L^{2}(\Gamma)$ and orthogonal in $H^{1}(\Gamma)$.  For example, we take $\{y_{j}\}_{j \in \N}$ as the set of eigenfunctions to the Laplace--Beltrami operator on $\Gamma$, i.e., 
\begin{align*}
-\LB y_{j} = \mu_{j} y_{j} & \text{ on } \Gamma
\end{align*}
with associated eigenvalue $\mu_{j}$.  Then, we define $Y_{n} := \mathrm{span}\{y_{1}, \dots, y_{n}\}$ as the finite-dimensional subspaces spanned by the first $n$ basis functions, with corresponding projection operator $\Pi_{Y_{n}}$.  For the bulk variable, consider the Hilbert spaces
\begin{align*}
\mathcal{H} := L^{2}(\Omega) \times L^{2}(\Gamma), \quad \mathcal{V} := \{(a,b) \in H^{1}(\Omega) \times H^{\frac{1}{2}}(\Gamma) \; : \; a \vert_{\Gamma} = b \},
\end{align*}
equipped with the inner products
\begin{align*}
((r_{1},r_{2}),(s_{1},s_{2}))_{\mathcal{H}} & = (r_{1},s_{1})_{L^{2}(\Omega)} + K^{-1} (r_{2},s_{2})_{L^{2}(\Gamma)}, \\
(\bm{p}, \bm{q})_{\mathcal{V}} & = (p,q)_{H^{1}(\Omega)} \text{ for } \bm{p} = (p, p\vert_{\Gamma}), \bm{q} = (q, q\vert_{\Gamma}) \in \mathcal{V}.
\end{align*}
By the continuity of the trace operator $\mathrm{tr}: H^{1}(\Omega) \to H^{\frac{1}{2}}(\Gamma)$, and the compact embedding $H^{\frac{1}{2}}(\Gamma) \subset \subset L^{2}(\Gamma)$, we easily infer that $\mathcal{V}$ compactly embeds into $\mathcal{H}$.  Let us define the abstract operator $\mathcal{A}: \mathcal{V} \to \mathcal{V}^{*}$ by
\begin{align*}
\inner{\mathcal{A}\bm{p}}{\bm{q}} := \int_{\Omega} \nabla p \cdot \nabla q \dx + \int_{\Gamma} K^{-1} p \vert_{\Gamma} \; q \vert_{\Gamma} \dHaus \quad \text{ for } \bm{p} = (p, p\vert_{\Gamma}), \bm{q} = (q, q\vert_{\Gamma}) \in \mathcal{V}.
\end{align*}
Then, $\mathcal{A}$ is non-negative, self-adjoint and, by the generalized Poincar\'{e} inequality, $\mathcal{A}$ is also coercive on $\mathcal{V}$.  Furthermore, for $\bm{h} = (h_{1},h_{2}) \in \mathcal{H}$, the abstract equation $\mathcal{A}\bm{p} = \bm{h}$ is equivalent to solving the Poisson problem
\begin{align*}
-\Lap p = h_{1} \text{ in } \Omega, \quad \pdnu p + K^{-1}p = K^{-1} h_{2} \text{ on } \Gamma.
\end{align*}
The Lax--Milgram theorem yields that $\mathcal{A}^{-1} : \mathcal{H} \to \mathcal{V} \subset \subset \mathcal{H}$ is a compact operator.  For arbitrary $\bm{h} = (h_{1},h_{2})$, $\bm{g} = (g_{1},g_{2}) \in \mathcal{H}$, denote by $\bm{p} = \mathcal{A}^{-1}\bm{h}$, $\bm{q} = \mathcal{A}^{-1}\bm{g} \in \mathcal{V}$.  Then, self-adjointness of $\mathcal{A}^{-1}$ can be easily seen from the chain of equalities
\begin{align*}
(\mathcal{A}^{-1}\bm{h}, \bm{g})_{\mathcal{H}} & = \int_{\Omega} p g_{1} \dx + \int_{\Gamma} K^{-1} p \vert_{\Gamma} \; g_{2} \dHaus = \int_{\Omega} \nabla p \cdot \nabla q \dx +  \int_{\Gamma} K^{-1} p \vert_{\Gamma} \; q \vert_{\Gamma} \dHaus \\
& = \int_{\Omega} h_{1} q \dx + \int_{\Gamma} K^{-1} h_{2} \; q \vert_{\Gamma} \dHaus = (\bm{h}, \mathcal{A}^{-1}\bm{g})_{\mathcal{H}}.
\end{align*}
Hence, by the theory of compact operators, we obtain a countable set of eigenvalues $\{\lambda_{i}\}_{i \in \N}$ with corresponding eigenfunctions $\{w_{i}\}_{i \in \N} \subset \mathcal{V}$ to the eigenvalue problem
\begin{align*}
\inner{\mathcal{A}\bm{p}}{\bm{q}} = \lambda (\bm{p},\bm{q})_{\mathcal{H}} \quad \forall \bm{q} \in \mathcal{V},
\end{align*}
where $\{w_{i}\}_{i \in \N}$ forms an orthonormal basis of $\mathcal{H}$ and an orthogonal basis of $\mathcal{V}$.  In particular, the above eigenvalue problem translates to the following strong form
\begin{align*}
-\Lap w_{i} = \lambda_{i} w_{i} \text{ in } \Omega, \quad \pdnu w_{i} + K^{-1}w_{i} = \lambda_{i} K^{-1} w_{i} \text{ on } \Gamma.
\end{align*}
Together with the orthonormality in $\mathcal{H}$, i.e., $((w_{i}, w_{i} \vert_{\Gamma}), (w_{j}, w_{j} \vert_{\Gamma}))_{\mathcal{H}} = \delta_{ij}$ where $\delta_{ij}$ is the Kronecker delta, we obtain
\begin{align*}
\int_{\Omega} \nabla w_{i} \cdot \nabla w_{j} \dx + \int_{\Gamma} K^{-1} w_{i} w_{j} \dHaus = \lambda_{i} \delta_{ij}.
\end{align*}
Consider now the finite-dimensional subspace $W_{n} = \mathrm{span}\{w_{1}, \dots, w_{n}\}$ of $H^{1}(\Omega)$ spanned by the first $n$ eigenfunctions, and denote the associated projection operator as $\Pi_{W_{n}}$.  For our solution we seek functions
\begin{align*}
u_{n}^{\eps} := \sum_{i=1}^{n} a_{i}^{n}(t) w_{i}(x) \in W_{n}, \quad \phi_{n}^{\eps} := \sum_{j=1}^{n} b_{j}^{n}(t) y_{j}(x) \in Y_{n}
\end{align*}
to the following system
\begin{align*}
0 & = \int_{\Omega} \pd_{t}u_{n}^{\eps} w_{k} + \nabla u_{n}^{\eps} \cdot \nabla w_{k} + (\beta_{\eps}(u_{n}^{\eps}) + \pi_{n}(u_{n}^{\eps}) - f) w_{k} \dx + \int_{\Gamma} K^{-1}(u_{n}^{\eps} - h(\phi_{n}^{\eps})) w_{k} \dHaus, \\
0 & = \int_{\Gamma} \pd_{t}\phi_{n}^{\eps} y_{k} + \surf \phi_{n}^{\eps} \cdot \surf y_{k} + (\beta_{\Gamma,\eps}(\phi_{n}^{\eps}) + \pi_{\Gamma}(\phi_{n}^{\eps}) - f_{\Gamma} + h'(\phi_{n}^{\eps}) K^{-1}(h(\phi_{n}^{\eps}) - u_{n}^{\eps})) y_{k} \dHaus,
\end{align*}
for $k = 1, \dots, n$, with $u_{n}^{\eps}(0) := \Pi_{W_{n}}(u_{0})$ and $\phi_{n}^{\eps}(0) := \Pi_{Y_{n}}(\phi_{0})$.  This is equivalent to the system
\begin{equation}\label{ODE}
\begin{aligned}
0 & = \underline{M} (\bm{a}^{n})'(t) + \underline{D} \bm{a}^{n}(t) + \bm{L}_{\Omega}(\bm{a}^{n}(t),\bm{b}^{n}(t)), \quad \underline{M}_{ij} = \int_{\Omega} w_{i} w_{j} \dx, \quad \underline{D}_{ij} = \lambda_{i} \delta_{ij}, \\
0 & = (\bm{b}^{n})'(t) + \underline{S} \bm{b}^{n}(t) + \bm{L}_{\Gamma}(\bm{a}^{n}(t), \bm{b}^{n}(t)), \quad \underline{S}_{ij} = \int_{\Gamma} \surf y_{i} \cdot \surf y_{j} \dHaus,
\end{aligned}
\end{equation}
for vectors $\bm{a}^{n} = (a_{1}^{n}, \dots, a_{n}^{n})$ and $\bm{b}^{n} = (b_{1}^{n}, \dots, b_{n}^{n})$, and
\begin{align*}
(\bm{L}_{\Omega})_{k} & = \int_{\Omega} (\beta_{\eps}(u_{n}^{\eps}) + \pi_{n}(u_{n}^{\eps}) - f) w_{k} \dx - \int_{\Gamma} K^{-1} h(\phi_{n}^{\eps}) w_{k} \dHaus, \\
(\bm{L}_{\Gamma})_{k} & = \int_{\Gamma} (\beta_{\Gamma,\eps}(\phi_{n}^{\eps}) + \pi_{\Gamma}(\phi_{n}^{\eps}) - f_{\Gamma} + h'(\phi_{n}^{\eps}) K^{-1}(h(\phi_{n}^{\eps}) - u_{n}^{\eps})) y_{k} \dHaus.
\end{align*}
The matrix $\underline{M}$ is positive definite, since for any vector $\bm{z} = (z_{1}, \dots, z_{n})$ corresponding to $\zeta := \sum_{i=1}^{n} z_{i} w_{i} \in W_{n}$, it holds that
\begin{align*}
\underline{M} \bm{z} \cdot \bm{z} = \int_{\Omega} \sum_{i=1}^{n} z_{i} w_{i} \; \sum_{j=1}^{n} z_{j} w_{j} \dx = \norm{\zeta}_{L^{2}(\Omega)}^{2} \geq 0,
\end{align*}
and $\underline{M} \bm{z} \cdot \bm{z} = 0$ if and only if $\zeta = 0$ if and only if $\bm{z} = \bm{0}$.

By virtue of the Lipschitz continuity of $\beta_{\eps}$, $\pi$, $\beta_{\Gamma,\eps}$, $\pi_{\Gamma}$, $h$, $h'$ and the assumption that $f \in C^{0}(0,T;L^{2}(\Omega))$, $f_{\Gamma} \in C^{0}(0,T;L^{2}(\Gamma))$, the system of ordinary differential equations \eqref{ODE} in $2n$ unknowns has a right-hand side that is continuous in $\bm{a}^{n} = (a_{1}^{n}, \dots, a_{n}^{n})$, $\bm{b}^{n} = (b_{1}^{n}, \dots, b_{n}^{n})$ and in time.  Therefore, by the Cauchy--Peano theorem, for each $n \in \N$, there exists a $t_{n} \in (0,\infty]$ such that on $[0,t_{n})$ a local solution $(\bm{a}^{n}, \bm{b}^{n}) \in C^{1}([0,t_{n}); \R^{2n})$ exists.  For our computations below, it is more convenient to express \eqref{ODE} in the following form
\begin{align}
\pd_{t} u^{\eps}_{n} - \Lap u^{\eps}_{n} + \Pi_{W_{n}} \left ( \beta_{\eps}(u^{\eps}_{n}) + \pi(u^{\eps}_{n}) - f \right ) = 0 & \text{ in } \Omega, \label{Gal:bulk} \\
\pd_{t} \phi^{\eps}_{n} - \LB \phi^{\eps}_{n} + \Pi_{Y_{n}} \left ( \beta_{\Gamma,\eps}(\phi^{\eps}_{n}) + \pi_{\Gamma}(\phi^{\eps}_{n}) - f_{\Gamma} + h'(\phi^{\eps}_{n}) K^{-1}(h(\phi^{\eps}_{n}) - u^{\eps}_{n}) \right ) = 0 & \text{ on } \Gamma, \label{Gal:surf} \\
K \pdnu u^{\eps}_{n} + u^{\eps}_{n} - P_{n}( h(\phi^{\eps}_{n})) = 0 & \text{ on } \Gamma \label{Gal:Robin}, \\
u^{\eps}_{n}(0) = \Pi_{W_{n}}(u_{0}) \text{ in } \Omega, \quad \phi^{\eps}_{n}(0) = \Pi_{Y_{n}}(\phi_{0}) & \text{ on } \Gamma,
\end{align}
where $P_{n}$ is the second component of the projection operator from $\mathcal{H}$ to the subspace spanned by the vectors $(w_{i}, w_{i} \vert_{\Gamma})$, $i = 1, \dots, n$, i.e., 
\begin{align*}
\int_{\Gamma} P_{n}(h(\phi_{n}^{\eps})) w_{k} \dHaus = \int_{\Gamma} h(\phi_{n}^{\eps}) w_{k} \dHaus \quad \text{ for all } 1 \leq k \leq n.
\end{align*}

\subsection{Uniform estimates}\label{sec:UniEst}

In the following, the symbol $C$ will denote positive constants that are independent of $\eps$ and $n$, but can depend on $K$ and $T$.  Furthermore, we will drop the superscript $\eps$ on $u^{\eps}_{n}$ and $\phi^{\eps}_{n}$ for convenience, as it turns out that our a priori estimates are independent of $\eps$.

\paragraph{First estimate.}  Testing \eqref{Gal:bulk} with $\pd_{t}u_{n}$, \eqref{Gal:surf} with $\pd_{t} \phi_{n}$ and upon summing leads to
\begin{equation}\label{EnergyEst}
\begin{aligned}
& \frac{\dd}{\dt} E_n(t) + \int_{\Omega} \abs{\pd_{t} u_{n}}^{2} \dx + \int_{\Gamma} \abs{\pd_{t}\phi_{n}}^{2} \dHaus \\
& \quad = \int_{\Omega} f \pd_{t} u_{n} \dx + \int_{\Gamma} f_{\Gamma} \pd_{t}\phi_{n} \dHaus \\
& \quad \leq \frac{1}{2} \Big ( \norm{f}_{L^2(\Omega)}^2 + \norm{f_{\Gamma}}_{L^2(\Gamma)}^2 + \norm{\pd_{t} u_n}_{L^2(\Omega)}^2 + \norm{\pd_t \phi_n}_{L^2(\Gamma)}^2 \Big),
\end{aligned}
\end{equation}
where
\begin{align*}
E_n(t) & : = \int_{\Omega} \frac{1}{2} \abs{\nabla u_{n}(t)}^{2} + W_{\eps}(u_{n}(t)) \dx + \int_{\Gamma} \frac{1}{2} \abs{\surf \phi_{n}(t)}^{2} + W_{\Gamma,\eps}(\phi_{n}(t)) \dHaus \\
& \quad + \int_{\Gamma} \frac{1}{2K} \abs{u_{n}(t) - h(\phi_{n}(t))}^{2} \dHaus, \\
W_{\eps}(u_{n}) & := \hat{\beta}_{\eps}(u_{n}) + \hat{\pi}(u_{n}), \quad W_{\Gamma,\eps}(\phi_{n}) := \hat{\beta}_{\Gamma,\eps}(\phi_{n}) + \hat{\pi}_{\Gamma}(\phi_{n}).
\end{align*}
Thanks to \eqref{Yosida:ini} and the fact that $\norm{u_{n}(0)}_{H^{1}(\Omega)} \leq C \norm{u_{0}}_{H^{1}(\Omega)}$, $\norm{\phi_{n}(0)}_{H^{1}(\Gamma)} \leq C \norm{\phi_{0}}_{H^{1}(\Gamma)}$, we infer the uniform boundedness of the initial energy $E_n(0)$.  Then, using \eqref{ass:f} we have for any $T \in (0,\infty)$,
\begin{equation}\label{Apri:1a}
\begin{aligned}
& \norm{\nabla u_{n}}_{L^{\infty}(0,T ;L^{2}(\Omega))}^2 + \norm{\pd_{t}u_{n}}_{L^{2}(0,T ;L^{2}(\Omega))}^2 + \norm{W_{\eps}(u_{n})}_{L^{\infty}(0,T ;L^{1}(\Omega))} \\
& \quad +  \norm{\surf \phi_{n}}_{L^{\infty}(0,T ;L^{2}(\Gamma))}^2 + \norm{\pd_{t}\phi_{n}}_{L^{2}(0,T ;L^{2}(\Gamma))}^2 + \norm{W_{\Gamma,\eps}(\phi_{n})}_{L^{\infty}(0,T ;L^{1}(\Gamma))} \\
& \quad + K^{-1} \norm{u_{n} - h(\phi_{n})}_{L^{\infty}(0,T ;L^{2}(\Gamma))}^2 \leq C \Big (1 + \norm{f}_{L^2(Q)}^2 + \norm{f_{\Gamma}}_{L^2(\Sigma)}^2 \Big ) \leq C.
\end{aligned}
\end{equation}
Furthermore, for any $s > 0$ it holds that
\begin{equation}\label{Linfty:L2}
\begin{aligned}
\frac{1}{2} \norm{u_{n}(s)}_{L^{2}(\Omega)}^{2} & = \int_{0}^{s} \! \!\int_{\Omega} u_{n} \, \pd_{t} u_{n} \dx \dt + \frac{1}{2} \norm{u_{n}(0)}_{L^{2}(\Omega)}^{2} \\
& \leq \frac{1}{2} \norm{\pd_{t}u_{n}}_{L^{2}(0,s;L^{2}(\Omega))}^{2} + \frac{1}{2} \norm{u_{n}}_{L^{2}(0,s;L^{2}(\Omega))}^{2} + \frac{1}{2} \norm{u_{n}(0)}_{L^{2}(\Omega)}^{2},
\end{aligned}
\end{equation}
and by a Gronwall argument one obtains for any $T \in (0,\infty)$,
\begin{align}\label{Apri:1b}
\norm{u_{n}}_{L^{\infty}(0,T;L^{2}(\Omega))} + \norm{\phi_{n}}_{L^{\infty}(0,T;L^{2}(\Gamma))} \leq C.
\end{align}
\begin{remark}\label{rmk:Unif:K:est}
If the initial conditions $u_{0}$ and $\phi_{0}$ fulfil the assumption
\begin{align*}
\norm{u_{0} - h(\phi_{0})}_{L^{2}(\Gamma)}^{2} \leq CK,
\end{align*}
for a positive constant $C$ not depending on $K$, then it is clear that
\begin{align*}
\norm{u_{n}(0) - h(\phi_{n}(0))}_{L^{2}(\Gamma)}^{2} \leq CK,
\end{align*}
and thus the initial energy $E(u_{n}(0), \phi_{n}(0))$ is also bounded uniformly in $K$.  As a consequence, the a priori estimates \eqref{Apri:1a} and \eqref{Apri:1b} are also uniform in $K$.  This will be relevant in Section \ref{sec:limit} when we pass to the limit as $K \to 0$.
\end{remark}

\paragraph{Second estimate.} Taking the time derivative of \eqref{Gal:bulk} and testing with $\pd_{t}u_{n}$ leads to
\begin{equation}\label{Est:2a}
\begin{aligned}
& \frac{\dd}{\dt} \int_{\Omega} \frac{1}{2} \abs{\pd_{t} u_{n}}^{2} \dx + \int_{\Omega} \abs{\nabla \pd_{t} u_{n}}^{2} + \beta_{\eps}'(u_{n}) \abs{\pd_{t}u_{n}}^{2} \dx + \frac{1}{2}K^{-1} \int_{\Gamma} \abs{\pd_{t} u_{n}}^{2} \dHaus \\
& \quad \leq \norm{\pi'}_{L^{\infty}(\R)} \norm{\pd_{t} u_{n}}_{L^{2}(\Omega)}^{2} + \norm{\pd_{t}f}_{L^{2}(\Omega)} \norm{\pd_{t} u_{n}}_{L^{2}(\Omega)} \\
& \qquad + \frac{1}{2}K^{-1} \norm{h'}_{L^{\infty}(\R)}^{2} \norm{\pd_{t} \phi_{n}}_{L^{2}(\Gamma)}^{2}.
\end{aligned}
\end{equation}
Here we used the fact that $\pd_{t} u_{n} \in W_{n}$ so that $\Pi_{W_{n}}(\pd_{t} u_{n}) = \pd_{t} u_{n}$ and
\begin{align*}
\int_{\Omega} \Pi_{W_{n}} \left ( \beta_{\eps}'(u_{n}) \pd_{t} u_{n} \right ) \pd_{t} u_{n} \dx = \int_{\Omega}  \beta_{\eps}'(u_{n}) \pd_{t} u_{n} \Pi_{W_{n}} \left ( \pd_{t} u_{n} \right ) \dx = \int_{\Omega} \beta_{\eps}'(u_{n}) \abs{\pd_{t} u_{n}}^{2} \dx .
\end{align*}
Due to the convexity of $\hat{\beta}_{\eps}$ , its second derivative $\beta_{\eps}'$ is non-negative, and in addition, as \eqref{ass:f} and \eqref{Apri:1a} imply that the right-hand side of \eqref{Est:2a} is bounded in $L^{1}(0,T)$, we obtain that
\begin{align}\label{Apri:2a}
\norm{\pd_{t} u_{n}}_{L^{\infty}(0,T; L^{2}(\Omega)) \cap L^{2}(0,T; H^{1}(\Omega))} \leq C.
\end{align}
\begin{remark}\label{rem:approx:ini}
In the above, we set
\begin{align*}
\pd_{t} u_{n}(0) := \Lap u_{n}(0) - \Pi_{W_{n}} \left (\beta_{\eps}(u_{0}) + \pi(u_{0}) - f(0) \right ).
\end{align*}
By the orthonormality of the basis functions $\{w_{i}\}_{i \in \N}$ in $L^{2}(\Omega)$ and the compatibility condition in \eqref{ass:ini} we have 
\begin{align*}
\norm{\pd_{t} u_{n}(0)}_{L^{2}(\Omega)} & \leq \norm{\Lap u_{0}}_{L^{2}(\Omega)} + \norm{\beta_{\eps}(u_{0})}_{L^{2}(\Omega)} + \norm{\pi(u_{0})}_{L^{2}(\Omega)} + \norm{f(0)}_{L^{2}(\Omega)} \\
& \leq C \left ( 1 + \norm{f(0)}_{L^{2}(\Omega)} + \norm{u_{0}}_{H^{2}(\Omega)} \right ),
\end{align*} 
where in the above we used the fact that the Yosida approximation $\beta_{\eps}$ satisfies the property $\beta_{\eps}(u_{0}) \rightharpoonup \beta^{\circ}(u_{0})$ in $L^{2}(\Omega)$ as $\eps \to 0$, and so $\norm{\beta_\eps(u_0)}_{L^2(\Omega)} \leq C$ for all $\eps \in (0,1)$.  Hence $\norm{\pd_{t}u_{n}(0)}_{L^{2}(\Omega)}$ is uniformly bounded in $n$ and~$\eps$.  
\end{remark}
In a similar fashion, taking the time derivative of \eqref{Gal:surf} and testing with $\pd_{t} \phi_{n}$ leads to
\begin{equation}\label{Est:2b}
\begin{aligned}
& \frac{\dd}{\dt} \int_{\Gamma} \frac{1}{2} \abs{\pd_{t} \phi_{n}}^{2} \dHaus + \int_{\Gamma} \abs{\surf \pd_{t} \phi_{n}}^{2} + \beta_{\Gamma,\eps}'(\phi_{n}) \abs{\pd_{t} \phi_{n}}^{2} + K^{-1} \abs{\pd_{t} h(\phi_{n})}^{2} \dHaus \\
& \quad \leq \norm{\pi_{\Gamma}'}_{L^{\infty}(\R)} \norm{\pd_{t} \phi_{n}}_{L^{2}(\Gamma)}^{2} + \norm{\pd_{t} f_{\Gamma}}_{L^{2}(\Gamma)} \norm{\pd_{t} \phi_{n}}_{L^{2}(\Gamma)} \\
& \qquad + \norm{h'}_{L^{\infty}(\R)} K^{-1} \norm{\pd_{t} u_{n}}_{L^{2}(\Gamma)} \norm{\pd_{t} \phi_{n}}_{L^{2}(\Gamma)} \\
& \qquad + \norm{h''}_{L^{\infty}(\R)} K^{-1} \norm{h(\phi_{n}) - u_{n}}_{L^{2}(\Gamma)} \norm{\pd_{t} \phi_{n}}_{L^{4}(\Gamma)}^{2}.
\end{aligned}
\end{equation}
In two dimensions we have the Gagliardo--Nirenberg inequality 
\begin{align}\label{GNineq}
\norm{\theta}_{L^{4}(\Gamma)}^{2} \leq C \norm{\theta}_{L^{2}(\Gamma)} \norm{\surf \theta}_{L^{2}(\Gamma)} + C \norm{\theta}_{L^{2}(\Gamma)}^{2},
\end{align}
and hence, on account of the boundedness of $h(\phi_{n}) - u_{n}$ in $L^{\infty}(0,T;L^{2}(\Gamma))$, the last term on the right-hand side of \eqref{Est:2b} can be estimated as
\begin{align*}
\norm{h(\phi_{n}) - u_{n}}_{L^{2}(\Gamma)} \norm{\pd_{t} \phi_{n}}_{L^{4}(\Gamma)}^{2} \leq C \norm{\pd_{t}\phi_{n}}_{L^{2}(\Gamma)}^{2} + \frac{1}{2} \norm{\surf \pd_{t} \phi_{n}}_{L^{2}(\Gamma)}^{2}.
\end{align*}
Then, upon integrating in time and using that $\pd_{t} f_{\Gamma}, \pd_{t} \phi_{n}, \pd_{t} u_{n}$ are bounded in $L^{2}(0,T;L^{2}(\Gamma))$ we obtain
\begin{align}
\label{Apri:2b}
\norm{\pd_{t} \phi_{n}}_{L^{\infty}(0,T;L^{2}(\Gamma)) \cap L^{2}(0,T;H^{1}(\Gamma))} + \norm{\pd_{t} h(\phi_{n})}_{L^{2}(0,T;L^{2}(\Gamma))} \leq C,
\end{align}
where we set
\begin{align*}
\pd_{t} \phi_{n}(0) = \LB \phi_{n}(0)- \Pi_{Y_{n}} \left ( \beta_{\Gamma,\eps}(\phi_{0}) + \pi_{\Gamma}(\phi_{0}) - f_{\Gamma}(0) + h'(\phi_{0}) K^{-1}(h(\phi_{0}) - u_{0}) \right ),
\end{align*}
and use the orthonormality of the basis functions $\{y_{j}\}_{j \in \N}$ in $L^{2}(\Gamma)$ and a similar argument to Remark \ref{rem:approx:ini} to show that $\norm{\pd_{t} \phi_{n}(0)}_{L^{2}(\Gamma)} \leq C$.  In light of \eqref{Apri:2a}, \eqref{Apri:2b} and the trace theorem, by the relation $\pdnu u_{n} = K^{-1}(h(\phi_{n}) - u_{n})$ we find that
\begin{align}\label{pdnu:un:est}
\norm{\pdnu u_{n}}_{H^{1}(0,T;L^{2}(\Gamma))} \leq C.
\end{align}
\subsection{Passing to the limit $n \to \infty$}
We now pass to the limit $n \to \infty$.  Thanks to the uniform estimates \eqref{Apri:1a}, \eqref{Apri:1b}, \eqref{Apri:2a}, \eqref{Apri:2b}, and \eqref{pdnu:un:est}, we obtain a pair of limit functions $(u^{\eps},\phi^{\eps})$ satisfying
\begin{equation*}
\begin{alignedat}{5}
u_{n} & \to u^{\eps} && \text{ weakly-* in } && L^{\infty}(0,T;H^{1}(\Omega)), \\
\pd_{t}u_{n} & \to \pd_{t}u^{\eps} && \text{ weakly-* in } && L^{\infty}(0,T;L^{2}(\Omega)) \cap L^{2}(0,T;H^{1}(\Omega)), \\
u_{n} & \to u^{\eps} && \text{ strongly in } && C^{0}([0,T];L^{q}(\Omega)) \text{ for } q < 6, \text{ and a.e. in } Q, \\
\phi_{n} & \to \phi^{\eps} && \text{ weakly-* in } && L^{\infty}(0,T;H^{1}(\Gamma)), \\
\pd_{t}\phi_{n} & \to \pd_{t}\phi^{\eps} && \text{ weakly-* in } && L^{\infty}(0,T;L^{2}(\Gamma)) \cap L^{2}(0,T;H^{1}(\Gamma)), \\
\phi_{n} & \to \phi^{\eps} && \text{ strongly in } && C^{0}([0,T];L^{r}(\Omega)) \text{ for } r < \infty, \text{ and a.e. in } \Sigma,
\end{alignedat}
\end{equation*}
with $u^{\eps}(0) = u_{0}$ in $L^{2}(\Omega)$, $\phi^{\eps}(0) = \phi_{0}$ in $L^{2}(\Gamma)$, and
\begin{subequations}
\begin{alignat}{3}
0 & = \int_{\Omega} (\pd_{t}u^{\eps} + \beta_{\eps}(u^{\eps}) + \pi(u^{\eps}) - f) \zeta + \nabla u^{\eps} \cdot \nabla \zeta \dx + \int_{\Gamma} K^{-1} (u^{\eps} - h(\phi^{\eps})) \zeta \dHaus, \label{Gal:u:eps} \\
0 & = \int_{\Gamma} (\pd_{t} \phi^{\eps} + \beta_{\Gamma,\eps}(\phi^{\eps}) + \pi_{\Gamma}(\phi^{\eps}) - f_{\Gamma} + h'(\phi^{\eps}) K^{-1}(h(\phi^{\eps}) - u^{\eps})) \mu + \surf \phi^{\eps} \cdot \surf \mu \dHaus, \label{Gal:phi:eps}
\end{alignat}
\end{subequations}
for a.e. $t \in (0,T)$, and for all $\zeta \in H^{1}(\Omega)$ and $\mu \in H^{1}(\Gamma)$.  Furthermore, by weak/weak-* lower semicontinuity of the Bochner norms, it holds that there exists a positive constant $C$, independent of $\eps$ and $n$, such that
\begin{equation}\label{Apri:unif:n}
\begin{aligned}
& \norm{u^{\eps}}_{L^{\infty}(0,T;H^{1}(\Omega))} + \norm{\pd_{t}u^{\eps}}_{L^{\infty}(0,T;L^{2}(\Omega)) \cap L^{2}(0,T;H^{1}(\Omega))} \\
& \quad + \norm{\phi^{\eps}}_{L^{\infty}(0,T;H^{1}(\Gamma))} + \norm{\pd_{t} \phi^{\eps}}_{L^{\infty}(0,T;L^{2}(\Gamma)) \cap L^{2}(0,T;H^{1}(\Gamma))} \\
& \quad + K^{-\frac{1}{2}} \norm{u^{\eps} - h(\phi^{\eps})}_{L^{\infty}(0,T;L^{2}(\Gamma))} + \norm{\pdnu u^{\eps}}_{H^{1}(0,T;L^{2}(\Gamma))} \leq C.
\end{aligned}
\end{equation}

\subsection{Additional uniform estimates}
Aside from \eqref{Apri:unif:n}, we derive additional uniform estimates for the pair $(u^{\eps},\phi^{\eps})$.  Since $\beta_{\Gamma,\eps}$ is Lipschitz continuous with Lipschitz constant $\eps^{-1}$, we can consider $\mu = \beta_{\Gamma,\eps}(\phi^{\eps})$ in \eqref{Gal:phi:eps} and obtain
\begin{equation}\label{Est:3a}
\begin{aligned}
& \int_{\Gamma} \beta_{\Gamma,\eps}'(\phi^{\eps}) \abs{\surf \phi^{\eps}}^{2} + \frac{1}{2} \abs{\beta_{\Gamma,\eps}(\phi^{\eps})}^{2} \dHaus \\
& \quad \leq C \left (\norm{\pd_{t} \phi^{\eps}}_{L^{2}(\Gamma)}^{2} + \norm{\pi_{\Gamma}(\phi^{\eps})}_{L^{2}(\Gamma)}^{2} + \norm{h'}_{L^{\infty}(\R)}^{2} \norm{h(\phi^{\eps}) - u^{\eps}}_{L^{2}(\Gamma)}^{2} + \norm{f_{\Gamma}}_{L^{2}(\Gamma)}^{2} \right ).
\end{aligned}
\end{equation}
Using the embedding $H^{1}(0,T) \subset L^{\infty}(0,T)$ for $\norm{f_{\Gamma}}_{L^{2}(\Gamma)}$, Lipschitz continuity of $\pi_{\Gamma}(\cdot)$ and \eqref{Apri:unif:n}, the right-hand side of \eqref{Est:3a} is bounded in $L^{\infty}(0,T)$ for any $T \in (0,\infty)$.  This shows that $\beta_{\Gamma,\eps}(\phi^{\eps})$ is bounded in $L^{\infty}(0,T;L^{2}(\Gamma))$.  Then, viewing \eqref{Gal:phi:eps} as the variational formulation for the following elliptic equation:
\begin{align*}
-\LB \phi^{\eps} = -\pd_{t}\phi^{\eps} -  \beta_{\Gamma,\eps}(\phi^{\eps}) - \pi_{\Gamma}(\phi^{\eps}) + f_{\Gamma} - h'(\phi^{\eps}) K^{-1}(h(\phi^{\eps}) - u^{\eps}) \text{ on } \Gamma,
\end{align*}
with a right-hand side belonging to $L^{\infty}(0,T;L^{2}(\Gamma))$ for any $T \in (0,\infty)$, we obtain altogether
\begin{align}\label{Apri:3a}
\norm{\phi^{\eps}}_{L^{\infty}(0,T;H^{2}(\Gamma))} + \norm{\beta_{\Gamma,\eps}(\phi^{\eps})}_{L^{\infty}(0,T;L^{2}(\Gamma))} \leq C
\end{align}
for any $T \in (0,\infty)$.  Similarly, considering $\zeta = \beta_{\eps}(u^{\eps})$ in \eqref{Gal:u:eps} leads to
\begin{equation}\label{Est:3b}
\begin{aligned}
& \int_{\Omega} \beta_{\eps}'(u^{\eps}) \abs{\nabla u^{\eps}}^{2} + \frac{1}{2} \abs{\beta_{\eps}(u^{\eps})}^{2} \dx + \int_{\Gamma} K^{-1} u^{\eps} \beta_{\eps}(u^{\eps}) \dHaus \\
& \quad \leq C \left ( \norm{\pd_{t} u^{\eps}}_{L^{2}(\Omega)}^{2} + \norm{\pi(u^{\eps})}_{L^{2}(\Omega)}^{2} + \norm{f}_{L^{2}(\Omega)}^{2} \right ) \\
& \qquad + C K^{-1} \norm{h(\phi^{\eps})}_{L^{\infty}(\Gamma)} \norm{\beta_{\eps}(u^{\eps})}_{L^{1}(\Gamma)}.
\end{aligned}
\end{equation}
Let us now recall the resolvent $\mathcal{J}_{\eps}$ which is a Lipschitz operator (with Lipschitz constant 1) defined as $\mathcal{J}_{\eps} := (\mathrm{I} + \eps \beta)^{-1}$ so that for all $\eps > 0$,
\begin{align}\label{resolvent}
g = \mathcal{J}_{\eps}(g) + \eps \beta_{\eps}(g) \quad \text{ and } \quad \beta_{\eps}(g) \in  \beta \left ( \mathcal{J}_{\eps}(g) \right ).
\end{align}
Taking $\xi = \beta_{\eps}(u^{\eps})$ and $u = \mathcal{J}_{\eps}(u^{\eps})$ in \eqref{ass:beta}, for any $\delta > 0$ there exists a positive constant $C_{\delta} > 0$ such that
\begin{equation}\label{ass:beta:Yosi}
\begin{aligned}
\abs{\beta_{\eps}(u^{\eps})} & \leq \delta \, \mathcal{J}_{\eps}(u^{\eps}) \, \beta_{\eps}(u^{\eps}) + C_{\delta} \\
& \leq \delta \abs{u^{\eps}} \abs{\beta_{\eps}(u^{\eps})} + C_{\delta} = \delta \, u^{\eps} \, \beta_{\eps}(u^{\eps}) + C_{\delta}
\end{aligned}
\end{equation}
where the second inequality follows from the Lipschitz property of $\mathcal{J}_{\eps}$ and $\mathcal{J}_{\eps}(0) = 0$, and the subsequent equality follows from the fact that $u^{\eps}$ and $\beta_{\eps}(u^{\eps})$ have the same sign due to the monotonicity of $\beta_{\eps}$.  Hence, choosing $\delta$ sufficiently small, from \eqref{Est:3b} we obtain
\begin{equation}\label{Est:3b:Sim}
\begin{aligned}
& \norm{\beta_{\eps}(u^{\eps})}_{L^{2}}^{2} + \int_{\Gamma} u^{\eps} \, \beta_{\eps}(u^{\eps}) \dHaus \\
& \quad \leq C \left ( \norm{h(\phi^{\eps})}_{L^{\infty}(\Gamma)} + \norm{\pd_{t} u^{\eps}}_{L^{2}(\Omega)}^{2} + \norm{\pi(u^{\eps})}_{L^{2}(\Omega)}^{2} + \norm{f}_{L^{2}(\Omega)}^{2} \right ).
\end{aligned}
\end{equation}
Using that $\phi^{\eps}$ is bounded in $L^{\infty}(0,T;H^{2}(\Gamma))$ and hence also in $L^{\infty}(0,T;L^{\infty}(\Gamma))$, as well as $f \in L^{\infty}(0,T;L^{2}(\Omega))$ and the estimate \eqref{Apri:unif:n}, we infer that the right-hand side of \eqref{Est:3b:Sim} is bounded in $L^{\infty}(0,T)$ for any $T \in (0,\infty)$.  As $u^{\eps} \, \beta_{\eps}(u^{\eps})$ is non-negative (due to the monotonicity of $\beta_{\eps}$), this implies that $\beta_{\eps}(u^{\eps})$ is bounded in $L^{\infty}(0,T;L^{2}(\Omega))$ for any $T \in (0,\infty)$.  Then, viewing \eqref{Gal:u:eps} as the variational formulation of an elliptic equation for $u^{\eps}$ with normal derivative $\pdnu u^{\eps} = K^{-1}(h(\phi^{\eps}) - u^{\eps})$ belonging to $L^{\infty}(0,T;H^{\frac{1}{2}}(\Gamma))$ and a right-hand side belonging to $L^{\infty}(0,T;L^{2}(\Omega))$ for any $T \in(0,\infty)$, we infer from \cite[Theorem 3.2, p.~1.79]{Brezzi} (see also Theorem \ref{Brezzi:reg}) that $u^{\eps}$ is bounded in $L^{\infty}(0,T;H^2(\Omega))$ together with the uniform estimate
\begin{align}
\label{Apri:3b}
\norm{u^{\eps}}_{L^{\infty}(0,T;H^{2}(\Omega))} + \norm{\beta_{\eps}(u^{\eps})}_{L^{\infty}(0,T;L^{2}(\Omega))} \leq C.
\end{align}

\subsection{Passing to the limit $\eps \to 0$} 
The uniform estimates \eqref{Apri:unif:n}, \eqref{Apri:3a} and \eqref{Apri:3b} are sufficient to allow us to pass to the limit $\eps \to 0$ and obtain a pair of limit functions $(u, \phi)$ that inherits the regularities stated in Theorem \ref{thm:Exist} and a pair of functions $(\xi, \xi_{\Gamma})$ such that
\begin{equation*}
\begin{alignedat}{5}
\beta_{\eps}(u^{\eps}) & \to \xi  && \text{ weakly-* in } L^{\infty}(0,T;L^{2}(\Omega)), \\
\beta_{\Gamma,\eps}(\phi^{\eps})  & \to \xi_{\Gamma} && \text{ weakly-* in } L^{\infty}(0,T;L^{2}(\Gamma)),
\end{alignedat}
\end{equation*}
as $\eps \to 0$.  Furthermore, passing to the limit in \eqref{Gal:u:eps} and \eqref{Gal:phi:eps} shows that the quadruple $(u,\phi,\xi,\xi_{\Gamma})$ satisfies \eqref{ACAC:gen}.  It remains to show that $\xi \in \beta(u)$ a.e.~in $Q$ and $\xi_{\Gamma} \in \beta_{\Gamma}(\phi)$ a.e.~on $\Sigma$, and for this we refer the reader to the argument presented near the end of \cite[Section 5.2]{GL}.  The basic idea is to use the strong convergence of $u^{\eps}$ to $u$ in $L^{2}(0,T;L^{2}(\Omega))$, the identity $u^{\eps} = \mathcal{J}_{\eps} (u^{\eps}) + \eps \beta_{\eps}(u^{\eps})$ from \eqref{resolvent} and the boundedness of $\beta_{\eps}(u^{\eps})$ in $L^{2}(0,T;L^{2}(\Omega))$ to deduce that $\mathcal{J}_{\eps} (u^{\eps}) \to u$ strongly in $L^{2}(0,T;L^{2}(\Omega))$.  Then, as $\beta_{\eps}(u^{\eps}) \in \beta (\mathcal{J}_{\eps}(u_{\eps}))$, monotonicity of $\beta$ shows that for arbitrary $y \in D(\beta)$ and $z \in \beta(y)$,
\begin{align*}
\int_{Q} (z - \beta_{\eps}(u^{\eps})) (y - \mathcal{J}_{\eps}(u^{\eps})) \dx \dt \geq 0 \quad \underset{\eps \to 0}{\Longrightarrow} \quad \int_{Q} (z - \xi)(y - u) \dx \dt \geq 0.
\end{align*} 
By definition of the maximal monotonicity, this yields $\xi \in \beta(u)$ a.e.~in $Q$.  For further details, we refer the reader to the proof of \cite[Lemma 1.3(e)]{BCP}.

\section{Omega-limit set}\label{sec:longtime}

\subsection{Non-emptiness of the omega-limit set} 

Testing \eqref{ACAC:bulk} with $\pd_{t} u$ and \eqref{ACAC:surf} with $\pd_{t} \phi$, integrating in time leads to an analogous identity to \eqref{EnergyEst}.  Then applying Young's inequality leads to 
\begin{align*}
& \frac{d}{dt} \left ( \int_{\Omega} \frac{1}{2} \abs{\nabla u}^{2} + W(u) \dx + \int_{\Gamma} \frac{1}{2} \abs{\surf \phi}^{2} + W_{\Gamma}(\phi) + \frac{1}{2K} \abs{u - h(\phi)}^{2} \dHaus \right ) \\
& \quad + \int_{\Omega} \frac{1}{2} \abs{\pd_{t} u}^{2} \dx + \int_{\Gamma} \frac{1}{2} \abs{\pd_{t} \phi}^{2} \dHaus \leq \int_{\Omega} \frac{1}{2} \abs{f}^{2} \dx + \int_{\Gamma} \frac{1}{2} \abs{f_{\Gamma}}^{2} \dHaus.
\end{align*}
On account of \eqref{ass:long:W} and \eqref{ass:long:f}, we immediately infer from integrating the above inequality the following uniform-in-time estimates:
\begin{equation}\label{Lt:1}
\begin{aligned}
& \norm{u}_{L^{\infty}(0,\infty;H^{1}(\Omega))} + \norm{\phi}_{L^{\infty}(0,\infty;H^{1}(\Gamma))} + K^{-\frac{1}{2}} \norm{u - h(\phi)}_{L^{\infty}(0,\infty;L^{2}(\Gamma))} \\
& \quad + \norm{\pd_{t}u}_{L^{2}(0,\infty;L^{2}(\Omega))} + \norm{\pd_{t}\phi}_{L^{2}(0,\infty;L^{2}(\Gamma))} \leq C.
\end{aligned}
\end{equation}
From this estimate the omega-limit set
\begin{equation}\label{omegalimit}
\begin{aligned}
\omega := \Big{\{} (u_{\infty}, \phi_{\infty}) & \, : \, \exists \{t_{k}\}_{k \in \N}, \, t_k > 0, \, t_{k} \nearrow \infty \\
& \text{ and } (u(t_{k}), \phi(t_{k})) \to (u_{\infty}, \phi_{\infty}) \text{ weakly in } H^{1}(\Omega) \times H^{1}(\Gamma) \Big{\}}
\end{aligned}
\end{equation}
is non-empty.  The aim of this section is to show that if $(u_{\infty}, \phi_{\infty}) \in \omega$ then $(u_{\infty}, \phi_{\infty})$ is a solution to the stationary problem \eqref{stationary}.

\subsection{Additional uniform-in-time estimates}

We require additional uniform estimates on the time interval $(0,\infty)$ before proceeding with the proof of Theorem \ref{thm:Lt}.  In light of the boundedness of $\pd_{t}u$ and $\pd_{t} \phi$ in $L^{2}(0,\infty;X)$ for $X = L^{2}(\Omega)$ and $X = L^{2}(\Gamma)$, respectively, it turns out that there are analogues of the estimates \eqref{Apri:2a}, \eqref{Apri:2b} and \eqref{pdnu:un:est} that hold on the time interval $(0,\infty)$.  This is summarized in the following lemma.

\begin{lemma}\label{lem:unif:time}
Under the hypothesis of Theorem \ref{thm:Lt}, the unique strong solution $(u,\phi)$ to \eqref{ACAC:gen} satisfies in addition to \eqref{Lt:1},
\begin{align*}
\pd_{t}u & \in L^{\infty}(0,\infty; L^{2}(\Omega)) \cap L^{2}(0,\infty;H^{1}(\Omega)), \quad \pdnu u \in H^{1}(0,\infty;L^{2}(\Gamma)), \\
\pd_{t}\phi & \in L^{\infty}(0,\infty;L^{2}(\Gamma)) \cap L^{2}(0,\infty;H^{1}(\Gamma)).
\end{align*}
\end{lemma}
\begin{proof}
We return to the two-level Galerkin approximation in Section \ref{sec:exist}, whereby thanks to \eqref{ass:long:W} and \eqref{ass:long:f} the Galerkin pair of solutions $(u_{n},\phi_{n})$ satisfies the uniform estimates \eqref{Apri:1a}, \eqref{Apri:1b} with $T = \infty$.  Denoting by the symbol $C$ positive constants not depending on $n$ and $\eps$ and using that $\pd_{t}f, \pd_{t} u_{n} \in L^{2}(0,\infty;L^{2}(\Omega))$, $\pd_{t}\phi_{n} \in L^{2}(0,\infty;L^{2}(\Gamma))$, the right-hand side of  \eqref{Est:2a} is bounded in $L^{1}(0,\infty)$, which yields
\begin{align*}
\norm{\pd_{t}u_{n}}_{L^{\infty}(0,\infty;L^{2}(\Omega)) \cap L^{2}(0,\infty;H^{1}(\Omega))} \leq C.
\end{align*}
In addition, using that $\pd_{t}f_{\Gamma}, \pd_{t}\phi_{n}, \pd_{t}u_{n} \in L^{2}(0,\infty;L^{2}(\Gamma))$ and $h(\phi_{n}) - u_{n} \in L^{\infty}(0,\infty;L^{2}(\Gamma))$, the right-hand side of \eqref{Est:2b} is bounded in $L^{1}(0,\infty)$.  This gives 
\begin{align*}
\norm{\pd_{t}\phi_{n}}_{L^{\infty}(0,\infty;L^{2}(\Gamma)) \cap L^{2}(0,\infty;H^{1}(\Gamma))} + \norm{\pd_{t}h(\phi_{n})}_{L^{2}(0,\infty;L^{2}(\Gamma))} \leq C.
\end{align*}
Lastly, using the trace theorem and the relation $K \pdnu u_{n} = h(\phi_{n}) - u_{n}$ we find that
\begin{align*}
\norm{\pdnu u_{n}}_{H^{1}(0,\infty;L^{2}(\Gamma))} \leq C.
\end{align*}
Passing to the limit $n \to \infty$ and $\eps \to 0$ leads to the desired assertion.
\end{proof}

\subsection{Uniform translation estimates}
Fix $T > 0$ and consider the problem \eqref{ACAC:gen} in the time interval $(t_{k}, T+ t_{k} )$, where $\{t_{k}\}_{k \in \N}$ is the sequence in \eqref{omegalimit}, and we introduce for $t \in [0,T]$ the functions
\begin{align*}
& u_{k}(t) := u(t_{k} + t), \quad \phi_{k}(t) : = \phi(t_{k} + t), \\
&\xi_{k}(t) := \xi(t_{k} + t), \quad \xi_{\Gamma,k}(t) := \xi_{\Gamma}(t_{k} + t) \quad \text{ a.e. in } [0,T], \\
&f_{k}(t) := f(t_{k} + t), \quad f_{\Gamma,k}(t) := f_{\Gamma}(t_{k} + t),
\end{align*}
which satisfy
\begin{subequations}\label{Lt:system}
\begin{alignat}{3}
\pd_{t} u_{k} = \Lap u_{k} - \xi_{k} - \pi(u_{k}) + f_{k} & \text{ in } Q, \label{Lt:bulk} \\
\pd_{t} \phi_{k} = \LB \phi_{k} - \xi_{\Gamma,k} - \pi_{\Gamma}(\phi_{k}) + f_{\Gamma,k} - h'(\phi_{k}) \pdnu u_{k} & \text{ on } \Sigma, \label{Lt:surf} \\
K \pdnu u_{k} + u_{k} = h(\phi_{k}) & \text{ on } \Sigma, \label{Lt:Robin} \\
\xi_{k} \in \beta(u_{k}) \text{ a.e. in } Q, \quad \xi_{\Gamma,k} \in \beta_{\Gamma}(\phi_{k}) \text{ a.e.} & \text{ on } \Sigma, \\
u_{k}(0) = u(t_{k}) \text{ in } \Omega, \quad \phi_{k}(0) = \phi(t_{k}) & \text{ on } \Gamma. \label{Lt:ini}
\end{alignat}
\end{subequations}
The aim is to derive uniform (in $k$) estimates and then pass to the limit $k \to \infty$. It turns out to be more convenient to work with the continuous solutions $(u^{\eps},\phi^{\eps})$ to the system \eqref{Lt:system} with $\beta$ and $\beta_{\Gamma}$ replaced by their corresponding Yosida approximations $\beta_{\eps}$ and $\beta_{\Gamma,\eps}$.  Therefore, instead of \eqref{Lt:system} we consider deriving uniform estimates in $k$ and $\eps$ for solutions $(u^{\eps}_k(t), \phi^{\eps}_k(t)) = (u^{\eps}(t_k + t), \phi^{\eps}(t_k + t))$ to
\begin{subequations}\label{Lt:system:eps}
\begin{alignat}{3}
\pd_t u^{\eps}_k = \Lap u^{\eps}_k - \beta_{\eps}(u^{\eps}_k) - \pi(u^{\eps}_k) + f_k & \text{ in } Q, \label{Lt:bulk:eps}\\
\pd_t \phi^{\eps}_k = \LB \phi^{\eps}_k - \beta_{\Gamma,\eps}(\phi^{\eps}_k) - \pi_{\Gamma}(\phi^{\eps}_k) + f_{\Gamma,k} - h'(\phi^{\eps}_k) \pdnu u^{\eps}_k & \text{ on } \Sigma, \label{Lt:surf:eps}  \\
K \pdnu u^{\eps}_k + u^{\eps}_k = h(\phi^{\eps}_k) & \text{ on } \Sigma, \label{Lt:Robin:eps} \\
u^{\eps}_k(0) = u^{\eps}(t_k) \text{ in } \Omega, \quad \phi^{\eps}_k(0) = \phi^{\eps}(t_k) & \text{ on } \Gamma,
\end{alignat}
\end{subequations}
and passing to the limit first $\eps \to 0$ and then $k \to \infty$.  Below the symbol $C$ denotes positive constants not depending on $k$ and $\eps$.

\paragraph{First estimate.} 
Recalling that $u^{\eps}(0) = u_0$ and $\phi^{\eps}(0) = \phi_0$, then via an analogous derivation to \eqref{Lt:1}, we deduce that the same estimates also hold for $(u^{\eps}, \phi^{\eps})$ with a right-hand side $C$ depending only on $(u_0, \phi_0, f, f_{\Gamma})$, i.e.,
\begin{equation}\label{Lt:1:eps}
\begin{aligned}
& \norm{u^{\eps}}_{L^{\infty}(0,\infty;H^{1}(\Omega))} + \norm{\phi^{\eps}}_{L^{\infty}(0,\infty;H^{1}(\Gamma))} + K^{-\frac{1}{2}} \norm{u^{\eps} - h(\phi^{\eps})}_{L^{\infty}(0,\infty;L^{2}(\Gamma))} \\
& \quad + \norm{\pd_{t}u^{\eps}}_{L^{2}(0,\infty;L^{2}(\Omega))} + \norm{\pd_{t}\phi^{\eps}}_{L^{2}(0,\infty;L^{2}(\Gamma))} \leq C.
\end{aligned}
\end{equation}
This implies that the sequence of initial data $\{(u^{\eps}_k(0), \phi^{\eps}_k(0))\}_{k \in \N} = \{(u^{\eps}(t_k), \phi^{\eps}(t_k))\}_{k \in \N}$ is uniformly bounded in $H^1(\Omega) \times H^1(\Gamma)$.  Then, by a similar procedure applied to \eqref{Lt:system:eps} we obtain for $(u^{\eps}_k, \phi^{\eps}_k)$:
\begin{equation}
\begin{aligned}\label{Lt:2}
& \norm{u^{\eps}_{k}}_{L^{\infty}(0,T;H^{1}(\Omega))} + \norm{\phi^{\eps}_{k}}_{L^{\infty}(0,T;H^{1}(\Gamma))} + K^{-\frac{1}{2}} \norm{u^{\eps}_{k} - h(\phi^{\eps}_{k})}_{L^{\infty}(0,T;L^{2}(\Gamma))} \\
& \quad + \norm{\pd_{t} u^{\eps}_{k}}_{L^{2}(0,T;L^{2}(\Omega))} + \norm{\pd_{t} \phi^{\eps}_{k}}_{L^{2}(0,T;L^{2}(\Gamma))} \leq C.
\end{aligned}
\end{equation}
Applying Lebesgue's dominated convergence theorem, and using the fact that $\pd_{t}u^{\eps} \in L^{2}(0,\infty;L^{2}(\Omega))$, $\pd_{t} \phi^{\eps} \in L^{2}(0,\infty ; L^{2}(\Gamma))$ from \eqref{Lt:1:eps} shows that
\begin{align}\label{Lt:2:zero}
\pd_{t} u^{\eps}_{k} \to 0 \text{ strongly in } L^{2}(0,T;L^{2}(\Omega)), \quad \pd_{t} \phi^{\eps}_{k} \to 0 \text{ strongly in } L^{2}(0,T;L^{2}(\Gamma)).
\end{align}
Indeed, we have by definition of $u^{\eps}_{k}$:
\begin{align*}
\norm{\pd_{t} u^{\eps}_{k}}_{L^{2}(0,T;L^{2}(\Omega))}^{2} = \int_{\R} \norm{ \pd_{t} u^{\eps}}_{L^{2}(\Omega)}^{2} \chi_{[t_{k}, t_{k}+T]}(t) \dt \to 0 \text{ as } \eps \to 0, \, k \to \infty.
\end{align*}

\paragraph{Second estimate.} 
In the proof of Lemma \ref{lem:unif:time}, by passing to the limit $n \to \infty$ for $\eps$ fixed and employing weak/weak-* lower semicontinuity of the Bochner norms we can deduce that 
\begin{align*}
\norm{\pd_t u^{\eps}}_{L^{\infty}(0,\infty;L^2(\Omega)) \cap L^2(0,\infty;H^1(\Omega))} + \norm{\pd_t \phi^{\eps}}_{L^{\infty}(0,\infty;L^2(\Gamma)) \cap L^2(0,\infty;H^1(\Gamma))} \leq C.
\end{align*}
Then, the fact that $t_k > 0$ immediately implies
\begin{align*}
\norm{\pd_t u^{\eps}_k}_{L^{\infty}(0,T;L^2(\Omega))} = \norm{\pd_t u^{\eps}}_{L^{\infty}(t_k, T + t_k ; L^2(\Omega))} \leq \norm{\pd_t u^{\eps}}_{L^{\infty}(0,\infty;L^2(\Omega))} \leq C, \\
\norm{\pd_t u^{\eps}_k}_{L^2(0,T;H^1(\Omega))} = \norm{\pd_t u^{\eps}}_{L^{2}(t_k, T+t_k; H^1(\Omega))} \leq \norm{\pd_t u^{\eps}}_{L^2(0,\infty;H^1(\Omega))} \leq C,
\end{align*}
and so we obtain the uniform estimates
\begin{align}
\norm{\pd_{t} u^{\eps}_{k}}_{L^{\infty}(0,T;L^{2}(\Omega)) \cap L^{2}(0,T;H^{1}(\Omega))} & \leq C, \label{Lt:3} \\
\norm{\pd_{t}\phi^{\eps}_{k}}_{L^{\infty}(0,T;L^{2}(\Gamma)) \cap L^{2}(0,T;H^{1}(\Gamma))}&  \leq C. \label{Lt:4}
\end{align}

\paragraph{Third estimate.} 
We can write \eqref{Lt:surf:eps} as an elliptic equation for $\phi^{\eps}_{k}$:
\begin{align}\label{Lt:ellip}
- \LB \phi^{\eps}_{k} + \beta_{\Gamma,\eps}(\phi^{\eps}_{k}) = - \pi_{\Gamma}(\phi_{k}) + f_{\Gamma,k} - \pd_{t} \phi^{\eps}_{k} - h'(\phi^{\eps}_{k}) K^{-1} (h(\phi^{\eps}_{k}) - u^{\eps}_{k}) \text{ on } \Gamma,
\end{align}
where the right-hand side is bounded in $L^{\infty}(0,T;L^{2}(\Gamma))$.  Indeed, by \eqref{Lt:4} the term $\pd_{t} \phi^{\eps}_{k}$ is bounded in $L^{\infty}(0,T;L^{2}(\Gamma))$, while $f_{\Gamma,k}$ is bounded in $L^{\infty}(0,T;L^{2}(\Gamma))$ due to \eqref{ass:long:f} and Morrey's inequality $H^{1}(0,\infty) \subset C^{0,\frac{1}{2}}(0,\infty) \subset L^{\infty}(0,\infty)$.  Furthermore, by the Lipschitz continuity of $\pi_{\Gamma}$, the boundedness of $h'$ and \eqref{Lt:2}, we see that
\begin{align*}
& \norm{\pi_{\Gamma}(\phi^{\eps}_{k})}_{L^{\infty}(0,T;L^{2}(\Gamma))} \leq L_{\pi_{\Gamma}} \norm{\phi^{\eps}_{k}}_{L^{\infty}(0,T;L^{2}(\Gamma))} + T \abs{\Gamma} \abs{\pi_{\Gamma}(0)} \leq C, \\
& \norm{h'(\phi^{\eps}_{k}) (h(\phi^{\eps}_{k}) - u^{\eps}_{k})}_{L^{\infty}(0,T;L^{2}(\Gamma))} \leq \norm{h(\phi^{\eps}_{k}) - u^{\eps}_{k}}_{L^{\infty}(0,T;L^{2}(\Gamma))} \norm{h'}_{L^{\infty}(\R)} \leq C.
\end{align*}
Then, testing \eqref{Lt:ellip} with $\beta_{\Gamma,\eps}(\phi^{\eps}_{k})$ and exploiting the non-negativity of $\beta_{\Gamma,\eps}'$ after integrating by parts, we infer that
\begin{align*}
\beta_{\Gamma,\eps}(\phi^{\eps}_{k}) \in L^{\infty}(0,T;L^{2}(\Gamma)),
\end{align*}
and by virtue of elliptic regularity we obtain altogether
\begin{align}\label{Lt:5}
\norm{\phi^{\eps}_{k}}_{L^{\infty}(0,T;H^{2}(\Gamma))} + \norm{\beta_{\Gamma,\eps}(\phi^{\eps}_k)}_{L^{\infty}(0,T;L^{2}(\Gamma))} \leq C.
\end{align}
The above estimate implies that $h(\phi^{\eps}_{k}) \in L^{\infty}(0,T;H^{\frac{1}{2}}(\Gamma) \cap L^{\infty}(\Gamma))$, and testing \eqref{Lt:bulk:eps} by $\beta_{\eps}(u^{\eps}_{k})$ yields
\begin{align*}
& \int_{\Omega} \beta_{\eps}'(u^{\eps}_{k}) \abs{\nabla u^{\eps}_{k}}^{2} + \abs{\beta_{\eps}(u^{\eps}_{k})}^{2} \dx + \int_{\Gamma} K^{-1} u^{\eps}_{k} \, \beta_{\eps}(u^{\eps}_{k}) \dHaus \\
& \quad \leq \left ( \norm{\pi(u^{\eps}_{k})}_{L^{2}(\Omega)} + \norm{f_{k}}_{L^{2}(\Omega)} \right ) \norm{\beta_{\eps}(u^{\eps}_{k})}_{L^{2}(\Omega)} \\
& \qquad + \norm{h(\phi^{\eps}_{k})}_{L^{\infty}(0,T;L^{\infty}(\Gamma))} \int_{\Gamma} K^{-1} \abs{\beta_{\eps}(u^{\eps}_{k})} \dHaus.
\end{align*}
Applying Young's inequality and \eqref{ass:beta:Yosi} we obtain that $\beta_{\eps}(u^{\eps}_{k}) \in L^{\infty}(0,T;L^{2}(\Omega))$.  Then, viewing \eqref{Lt:bulk:eps}, \eqref{Lt:Robin:eps} as an elliptic equation for $u^{\eps}_{k}$, by elliptic regularity (Theorem \ref{Brezzi:reg}) we obtain altogether
\begin{align}\label{Lt:6}
\norm{u^{\eps}_{k}}_{L^{\infty}(0,T;H^{2}(\Omega))} + \norm{\beta_{\eps}(u^{\eps}_{k})}_{L^{\infty}(0,T;L^{2}(\Omega))} \leq C.
\end{align}

\subsection{Passing to the limit}
Taking into account the uniform estimates \eqref{Lt:1:eps}, \eqref{Lt:2}, \eqref{Lt:3}, \eqref{Lt:4}, \eqref{Lt:5}, \eqref{Lt:6}, first sending $\eps \to 0$ and due to the uniqueness of solutions to \eqref{Lt:system} (cf.~Theorem \ref{thm:Ctsdep}), we can infer that the limit functions $(u_k, \phi_k, \xi_k, \xi_{\Gamma,k})$, with selections $\xi_k \in \beta(u_k)$ and $\xi_{\Gamma,k} \in \beta_{\Gamma}(\phi_k)$, satisfy the same uniform (in $k$) estimates as in \eqref{Lt:1:eps}, \eqref{Lt:2}, \eqref{Lt:3}, \eqref{Lt:4}, \eqref{Lt:5}, \eqref{Lt:6} thanks to weak/weak-* lower semicontinuity of the Bochner norms.  Hence, there exist functions $(u_{\infty}, \phi_{\infty}, \xi_{\infty}, \xi_{\Gamma,\infty})$ such that
\begin{equation*}
\begin{alignedat}{5}
u_{k} & \to u_{\infty} && \text{ weakly-* in } && L^{\infty}(0,T;H^{2}(\Omega)) \cap W^{1,\infty}(0,T;L^{2}(\Omega)) \cap H^{1}(0,T;H^{1}(\Omega)), \\
u_{k} & \to u_{\infty} && \text{ strongly in } && C^{0}([0,T];W^{1,p}(\Omega)) \text{ for } p < 6, \text{ and a.e. in } Q, \\
\phi_{k} & \to \phi_{\infty} && \text{ weakly-* in } && L^{\infty}(0,T;H^{2}(\Gamma)) \cap W^{1,\infty}(0,T;L^{2}(\Gamma)) \cap H^{1}(0,T;H^{1}(\Gamma)), \\
\phi_{k} & \to \phi_{\infty} && \text{ strongly in } && C^{0}([0,T];W^{1,q}(\Gamma)) \text{ for } q < \infty, \text{ and a.e. on } \Sigma, \\
\xi_{k} & \to \xi_{\infty} && \text{ weakly in } && L^{2}(0,T;L^{2}(\Omega)), \\
\xi_{\Gamma,k} & \to \xi_{\Gamma,\infty} && \text{ weakly in } && L^{2}(0,T;L^{2}(\Gamma)),
\end{alignedat}
\end{equation*}
with $u_{\infty}, \phi_{\infty}$ independent of time due to \eqref{Lt:2:zero}.  Note that $(u_{\infty}, \phi_{\infty})$ is exactly the element in \eqref{omegalimit} as it follows from passing to the limit in \eqref{Lt:ini}.  Furthermore, the strong convergence of $u_{k}$ and $\phi_{k}$ also allow us to deduce that $\xi_{\infty} \in \beta(u_{\infty})$ a.e. in $Q$, $\xi_{\Gamma,\infty} \in \beta_{\Gamma}(\phi_{\infty})$ a.e. on $\Sigma$.  Then, passing to the limit $k \to \infty$ in \eqref{Lt:system} shows that $(u_{\infty}, \phi_{\infty}, \xi_{\infty}, \xi_{\Gamma,\infty})$ satisfy
\begin{align*}
\Lap u_{\infty} - \xi_{\infty} - \pi(u_{\infty})  = 0 & \text{ in } \Omega, \\
\LB \phi_{\infty} - \xi_{\Gamma,\infty} - \pi_{\Gamma}(\phi_{\infty}) - h'(\phi_{\infty}) \pdnu u_{\infty} = 0 & \text{ on } \Gamma, \\
K \pdnu u_{\infty} + u_{\infty} = h(\phi_{\infty}) & \text{ on } \Gamma,
\end{align*}
where by comparison of terms we also deduce that $\xi_{\infty}$ and $\xi_{\Gamma,\infty}$ are time independent.

\section{Fast reaction limit}\label{sec:limit}

\subsection{Weak solutions}
For each $K > 0$, let $(u_{0,K}, \phi_{0,K}, f_{K}, f_{\Gamma,K})$ denote a set of data satisfying the assumptions in Theorem \ref{thm:limitK}.  Then, by Theorems \ref{thm:Exist} and \ref{thm:Ctsdep} there exists a corresponding unique strong solution $(u_{K}, \phi_{K})$  to \eqref{ACAC:gen}.  Furthermore, as outlined in Remark~\ref{rmk:Unif:K:est}, the a priori estimates \eqref{Apri:1a} and \eqref{Apri:1b} for $(u_{K},\phi_{K})$ are uniform in $K$, and so there exist a subsequence (not relabelled) and limit functions $(u,\phi)$ such that
\begin{align*}
u_{K} & \to u \text{ weakly-* in } L^{\infty}(0,T;H^{1}(\Omega)) \cap H^{1}(0,T;L^{2}(\Omega)), \\
u_{K} & \to u \text{ strongly in } C^{0}([0,T];L^{q}(\Omega)), \, q < 6, \text{ and a.e. in } Q, \\
\phi_{K} & \to \phi \text{ weakly-* in } L^{\infty}(0,T;H^{1}(\Gamma)) \cap H^{1}(0,T;L^{2}(\Gamma)), \\
\phi_{K} & \to \phi \text{ strongly in } C^{0}([0,T];L^{r}(\Gamma)), \, r < \infty, \text{ and a.e. on } \Sigma, \\
u_{K} - h(\phi_{K}) & \to 0 \text{ strongly in } L^{2}(0,T;L^{2}(\Gamma)).
\end{align*}
The last convergence shows that
\begin{align*}
u \vert_{\Sigma} = h(\phi) \text{ a.e. on } \Sigma.
\end{align*}
Using \eqref{Kto0:beta} we find that
\begin{align*}
\int_{\Omega} \abs{\beta(u_{K})}^{\frac{6}{p}} \dx \leq C \left ( 1 + \int_{\Omega} \abs{u_{K}}^{6} \dx \right ), \quad \int_{\Gamma} \abs{\beta_{\Gamma}(\phi_{K})}^{s} \dHaus \leq C \left ( 1 + \int_{\Gamma} \abs{\phi_{K}}^{qs} \dHaus \right )
\end{align*}
for any $s < \infty$.  This shows that for any $m < \infty$, $r < \frac{6}{p}$ and $s < \infty$,
\begin{align*}
\beta(u_{K}) & \to \beta(u) \text{ weakly-* in } L^{\infty}(0,T;L^{\frac{6}{p}}(\Omega)) \text{ and strongly in } L^{m}(0,T;L^{r}(\Omega)), \\
\beta_{\Gamma}(\phi_{K}) & \to \beta_{\Gamma}(\phi) \text{ weakly-* in } L^{\infty}(0,T;L^{s}(\Gamma)) \text{ and strongly in } L^{m}(0,T;L^{s}(\Gamma)).
\end{align*}
The assertions of strong convergence come from a.e.~convergence (as $\beta$ and $\beta_{\Gamma}$ are single-valued and continuous) and the application of Egorov's theorem.
Testing \eqref{ACAC:bulk} with an arbitrary test function $\zeta$ gives
\begin{align}\label{bulk:weak}
\int_{\Omega} \pd_{t} u_{K} \, \zeta + \nabla u_{K} \cdot \nabla \zeta + \beta(u_{K}) \, \zeta + \pi(u_{K}) \, \zeta - f_{K} \, \zeta \dx - \int_{\Gamma} \pdnu u_{K} \, \zeta \dHaus = 0.
\end{align}
Meanwhile, testing \eqref{ACAC:surf} with the test function $\frac{1}{h'(\phi_{K})} \zeta_{\Gamma} = \frac{1}{\alpha} \zeta_{\Gamma}$ leads to
\begin{align}\label{surf:weak}
\int_{\Gamma} \frac{1}{\alpha} \left (\pd_{t}\phi_{K} \, \zeta_{\Gamma} + \surf \phi_{K} \cdot \surf \zeta_{\Gamma} + \beta_{\Gamma}(\phi_{K}) \, \zeta_{\Gamma} + \pi_{\Gamma}(\phi_{K}) \, \zeta_{\Gamma} - f_{\Gamma,K} \, \zeta_{\Gamma} \right ) + \pdnu u_{K} \, \zeta_{\Gamma} \dHaus = 0.
\end{align}
We now consider an arbitrary test function $\zeta \in L^{2}(0,T;H^{1}(\Omega))$ such that $\zeta_{\Gamma} := \zeta \vert_{\Gamma} \in L^{2}(0,T;H^{1}(\Gamma))$.  Then, upon adding the equations \eqref{bulk:weak} and \eqref{surf:weak}, so that the terms involving $\pdnu u_{K}$ cancel, leads to
\begin{equation}\label{Kto0:combined:weak}
\begin{aligned}
0 & = \int_{0}^{T} \int_{\Omega} \pd_{t} u_{K} \, \zeta + \nabla u_{K} \cdot \nabla \zeta + \beta(u_{K}) \, \zeta + \pi(u_{K}) \, \zeta - f_{K} \, \zeta \dx \dt \\
& \quad + \int_{0}^{T} \int_{\Gamma} \frac{1}{\alpha} \left (\pd_{t}\phi_{K} \, \zeta_{\Gamma} + \surf \phi_{K} \cdot \surf \zeta_{\Gamma} + \beta_{\Gamma}(\phi_{K}) \, \zeta_{\Gamma} + \pi_{\Gamma}(\phi_{K}) \, \zeta_{\Gamma} - f_{\Gamma,K} \, \zeta_{\Gamma} \right ) \dHaus \dt.
\end{aligned}
\end{equation}
Passing to the limit $K \to 0$ shows that the limit functions $(u, \phi)$ satisfy \eqref{Kto0:weakform}.  Let us mention that the restriction $p \leq 5$ on the growth of $\beta$ is due to the fact that the product $\beta(u) \, \zeta$ is integrable for $\zeta \in L^{6}(\Omega)$ if and only if $\beta(u) \in L^{\frac{6}{5}}(\Omega)$.

\bigskip

The proof of Theorem \ref{thm:limitK} for the case of maximal monotone graphs is similar.  Thanks to assumption \eqref{Kto0:maxmono} the selections $\xi_{K} \in \beta(u_{K})$ and $\xi_{\Gamma,K} \in \beta_{\Gamma}(\phi_{K})$ are bounded in $L^{\infty}(0,T;L^{\frac{6}{p}}(\Omega))$ and in $L^{\infty}(0,T;L^{s}(\Gamma))$, respectively, for any $s < \infty$.  Then, there exists a subsequence (not relabelled) such that 
\begin{align*}
\xi_{K} & \to \xi \text{ weakly-* in } L^{\infty}(0,T;L^{\frac{6}{5}}(\Omega)), \\
\xi_{\Gamma,K} & \to \xi_{\Gamma} \text{ weakly-* in } L^{\infty}(0,T;L^{s}(\Gamma)).
\end{align*}
In order to show that $\xi \in \beta(u)$ a.e.~in $Q$ and $\xi_{\Gamma} \in \beta_{\Gamma}(\phi)$ a.e.~in $\Sigma$, it suffices to have $u_{K} \to u$ strongly in $L^{1}(0,T;L^{\frac{6}{6-p}}(\Omega))$ and $\phi_{K} \to \phi$ strongly in $L^{1}(0,T;L^{\frac{s}{s-1}}(\Gamma))$.  The strong convergence of $\phi_{K}$ is valid for any $s < \infty$, and for the strong convergence of $u_{K}$, we require $\frac{6}{6-p} < 6$ which is equivalent to $p < 5$.

\paragraph{Continuous dependence.} Let $\{(u_{i}, \phi_{i})\}_{i=1,2}$ denote two solutions to \eqref{Limit:form:h} corresponding to the data $\{(u_{0,i}, \phi_{0,i}, f_{i}, f_{\Gamma,i})\}_{i=1,2}$ and denote the difference by $\hat{u}$, $\hat{\phi}$, $\hat{u}_{0}$, $\hat{\phi}_{0}$, $\hat{f}$ and $\hat{f}_{\Gamma}$, respectively.  Then, substituting $\zeta = \hat{u}$ in the difference of \eqref{Kto0:weakform} and noting that
\begin{align*}
\zeta_{\Gamma} = \hat{u} \vert_{\Sigma} = h(\phi_{1}) - h(\phi_{2}) = \alpha \hat{\phi},
\end{align*}
we obtain
\begin{align*}
& \frac{1}{2} \frac{\dd}{\dt} \left ( \norm{\hat{u}}_{L^{2}(\Omega)}^{2} + \norm{\hat{\phi}}_{L^{2}(\Gamma)}^{2} \right ) + \norm{\nabla \hat{u}}_{L^{2}(\Omega)}^{2} + \norm{\surf \hat{\phi}}_{L^{2}(\Gamma)}^{2} \\
& \quad \leq \norm{\hat{f}}_{L^{2}(\Omega)} \norm{\hat{u}}_{L^{2}(\Omega)} + L_{\pi} \norm{\hat{u}}_{L^{2}(\Omega)}^{2} + \norm{\hat{f}_{\Gamma}}_{L^{2}(\Gamma)} \norm{\hat{\phi}}_{L^{2}(\Gamma)} + L_{\pi_{\Gamma}} \norm{\hat{\phi}}_{L^{2}(\Gamma)}^{2},
\end{align*}
where we have used the monotonicity of $\beta$ and $\beta_{\Gamma}$.  Then, by Gronwall's inequality we obtain the desired result.

\subsection{Strong solutions}

\subsubsection{Approximation scheme}
Following the approximation procedure of Colli and Fukao \cite{ColliFukaoAC}, for $\eps \in (0,1)$ we consider the approximation problem
\begin{equation}\label{Abst:Approx}
\begin{aligned}
\AA^{2} \bm{u}_{\eps}'(t) + \pd \varphi_{\eps}(\bm{u}_{\eps}(t)) + \AA \left (\bm{\pi}(\bm{u}_{\eps}(t))  - \bm{f}(t) \right ) \ni \bm{0} & \text{ in } \bm{H} \text{ for a.e. } t \in (0,T), \\
\bm{u}_{\eps}(0) = \bm{u}_{0} & \text{ in } \bm{H},
\end{aligned}
\end{equation}
where the function $\varphi_{\eps} : \bm{H} \to [0,\infty]$ is defined as
\begin{align}
\label{CF:eps:varphi}
\varphi_{\eps}(\bm{z}) = \begin{cases}
\displaystyle \int_{\Omega} \frac{1}{2} \abs{\nabla z}^{2} + \frac{\eps}{2} \abs{z}^{2} + \hat{\beta}_{\eps}(z) \dx  \\
\quad \displaystyle +  \int_{\Gamma} \frac{1}{2 \alpha^{2}} \abs{\surf z_{\Gamma}}^{2} + \frac{\eps}{2 \alpha^2} \abs{z_{\Gamma}}^{2} + \hat{\beta}_{\Gamma,\eps}(g(z_{\Gamma})) \dHaus \quad \text{ if } \bm{z} = (z, z_{\Gamma}) \in \bm{V}, \\
+ \infty \quad \text{ otherwise}.
\end{cases}
\end{align}
In the above, we recall the constant matrix $\AA$ is defined in \eqref{mat}, the product Hilbert spaces $\bm{H}$ and $\bm{V}$ are defined in \eqref{Prod:Space} equipped with inner products defined in \eqref{Prod:Space:innerProd}, $\beta_{\eps}$ and $\beta_{\Gamma,\eps}$ are the Yosida approximations of $\beta$ and $\beta_{\Gamma}$ with antiderivatives $\hat{\beta}_{\eps}$ and $\hat{\beta}_{\Gamma,\eps}$, respectively, and $g(s) = \alpha^{-1}(s - \eta)$ for $\alpha \neq 0$, $\eta \in \R$.

Note that the composition of a convex function with an affine linear function is convex, and thus, thanks to \cite[Lemma 3.1]{ColliFukaoAC}, the function $\varphi_{\eps} : \bm{H} \to [0,\infty]$ is convex, lower semicontinuous with domain $D(\varphi_{\eps}) = \bm{V}$.   Furthermore, $\varphi_{\eps}$ is lower semicontinuous in $\bm{V}$ and the subdifferential $\pd_{*} \varphi_{\eps}$ as an operator mapping from $\bm{V}$ to its dual space $\bm{V}'$ is single-valued and is characterized by the following:
\begin{equation}\label{pd:varphi:eps}
\begin{aligned}
\inner{\pd_{*} \varphi_{\eps}(\bm{z})}{\bm{y}}_{\bm{V}', \bm{V}} & = \int_{\Omega} \nabla z \cdot \nabla y + \eps z \, y + \beta_{\eps}(z) \, y \dx \\
&\quad  + \int_{\Gamma} \alpha^{-2} \surf z_{\Gamma} \cdot \surf y_{\Gamma} + \alpha^{-2} \eps z_{\Gamma} \, y_{\Gamma} + \alpha^{-1} \beta_{\Gamma,\eps}(g(z_{\Gamma}))  \, y_{\Gamma} \dHaus
\end{aligned}
\end{equation}
for all $\bm{z} = (z, z_{\Gamma}), \bm{y} = (y, y_{\Gamma}) \in \bm{V}$.  Note that the extra factor $\alpha^{-1}$ appearing before $\beta_{\Gamma,\eps}(g(z_{\Gamma}))$ in \eqref{pd:varphi:eps} arises from the derivative of $g$.

\begin{lemma}\label{lem:eps:wellpose}
For each $\eps \in (0,1]$, there exists a unique
\begin{align*}
\bm{u}_{\eps} := (u_{\eps}, u_{\Gamma,\eps}) \in H^{1}(0,T;\bm{H}) \cap L^{\infty}(0,T;\bm{V})
\end{align*}
satisfying \eqref{Abst:Approx}.
\end{lemma}

\begin{proof}
We sketch the basic steps:
\paragraph{Step 1.} For a given $\bm{w} \in C^{0}([0,T];\bm{H})$, the equation
\begin{align*}
\AA^{2} \bm{u}'(t) +  \pd \varphi_{\eps}(\bm{u}(t)) \ni \AA ( \bm{f}(t) - \bm{\pi}(\bm{w}(t))) & \text{ in } \bm{H}, \\
\bm{u}(0) = \bm{u}_{0} & \text{ in } \bm{H}
\end{align*}
admits a unique solution $\bm{u} \in L^{\infty}(0,T;\bm{V}) \cap H^{1}(0,T;\bm{H})$ for any $\bm{u}_{0} \in \bm{V}$.  Indeed, the right-hand side belongs to $L^{2}(0,T;\bm{H})$, and setting $A := \AA^{2}$, $B := \pd \varphi_{\eps}$, we can applying \cite[Theorem 2.1]{ColliVisintin} to deduce the existence result.  This is thanks to the fact that $A$ is a constant operator, and $B$ is the subdifferential of a proper, convex and lower semicontinuous function mapping $\bm{H}$ to $[0,\infty]$.  Furthermore, by \eqref{CF:eps:varphi}, it holds that $\varphi_{\eps}(\bm{z}) \geq C(\alpha, \eps) \norm{\bm{z}}_{\bm{V}}^{2}$, and so the assumptions of \cite[p.~741]{ColliVisintin} are fulfilled.  Concerning uniqueness,  we regard $A = \AA^{2}$ as a linear, self-adjoint operator in $\bm{H}$ that is strictly monotone.  Then, uniqueness of solutions is given by \cite[Remark 2.5]{ColliVisintin}, and from this we can construct a map $\Psi : \bm{w} \mapsto \bm{u}$ from $C^{0}([0,T];\bm{H})$ into itself.

\paragraph{Step 2.} Given a pair $\bm{w}_{1}, \bm{w}_{2} \in C^{0}([0,T];\bm{H})$ with corresponding solutions $\bm{u}_{1}, \bm{u}_{2}$, we find that by testing the difference of the equations with $\bm{u}_{1} - \bm{u}_{2}$, using the monotonicity of $\pd \varphi_{\eps}$ (see \eqref{pd:varphi:eps}), and integrating in time leads to
\begin{align*}
\norm{\bm{u}_{1}(s) - \bm{u}_{2}(s)}_{\bm{H}}^{2} \leq C_{\alpha, \bm{\pi}} \int_{0}^{s} \norm{\bm{w}_{1}(t) - \bm{w}_{2}(t)}_{\bm{H}}^{2} \dt \quad \forall s \in [0,T],
\end{align*}
for some positive constant $C_{\alpha, \bm{\pi}}$ depending only on $\alpha$ and the Lipschitz constants of $\pi$ and $\pi_{\Gamma}$, since both $\bm{u}_{1}$ and $\bm{u}_{2}$ have the same initial data and forcing.  Let us mention that by the monotonicity of $\beta_{\Gamma,\eps}$ and the affine linear relation $g$, we have
\begin{align*}
& \int_{\Gamma} \left ( \beta_{\Gamma,\eps}(g(u_{\Gamma,1})) - \beta_{\Gamma,\eps}(g(u_{\Gamma,2})) \right ) \left ( \alpha^{-1} (u_{\Gamma,1} - u_{\Gamma,2}) \right ) \dHaus \\
& \quad = \int_{\Gamma} \left ( \beta_{\Gamma,\eps}(g(u_{\Gamma,1})) - \beta_{\Gamma,\eps}(g(u_{\Gamma,2})) \right ) \left ( g(u_{\Gamma,1}) - g(u_{\Gamma,2}) \right ) \dHaus \geq 0.
\end{align*}
Choosing $s_{*} \in (0,T]$ such that $C_{\alpha, \bm{\pi}} s_{*} < 1$, the above estimate shows that $\Psi: C^{0}([0,s_{*}];\bm{H})$ into itself is a contraction, and by the  contraction mapping principle, $\Psi$ has a unique fixed point and thus there exists a unique solution $\bm{u}_{\eps}$ on the interval $[0,s_{*}]$ to \eqref{Abst:Approx}.  

Thanks to the fact that $C_{\alpha, \bm{\pi}}$ is independent of the initial values, we solve \eqref{Abst:Approx} on the interval $[s_{*}, 2 s_{*}]$ by setting $\bm{u}(s_{*})$ as the new initial value.  The above arguments yield that $\Psi: C^{0}([s_{*}, 2 s_{*}]; \bm{H})$ into itself is a contraction and this allows us to extend the unique solution $\bm{u}_{\eps}$ to the interval $[s_{*}, 2 s_{*}]$.  Iterating the process a finite number of times leads to the desired assertion on $[0,T]$.

\end{proof}

Thanks to Lemma \ref{lem:eps:wellpose} we see that $\bm{u}_{\eps} = (u_{\eps}, u_{\Gamma,\eps})$ satisfies the following variational formulation:
\begin{equation}\label{eps:weakform}
\begin{aligned}
0 & = \int_{\Omega} \nabla u_{\eps} \cdot \nabla \zeta \dx +  \int_{\Gamma} \alpha^{-2} \surf u_{\Gamma,\eps} \cdot \surf \zeta_{\Gamma} \dHaus \\
& \qquad + \int_{\Omega} \left ( \pd_{t} u_{\eps} + \eps u_{\eps} + \beta_{\eps}(u_{\eps}) + \pi(u_{\eps}) - f \right ) \, \zeta \dx \\
& \qquad + \alpha^{-1} \int_{\Gamma}  \left ( \alpha^{-1} \pd_{t} u_{\Gamma,\eps} + \alpha^{-1} \eps u_{\Gamma,\eps} + \beta_{\Gamma,\eps}(g(u_{\eps})) + \pi_{\Gamma}(g(u_{\eps})) - f_{\Gamma} \right ) \, \zeta_{\Gamma} \dHaus 
\end{aligned}
\end{equation}
for all $\bm{\zeta} = (\zeta, \zeta_{\Gamma}) \in \bm{V}$.

\bigskip

We now derive regularity properties for the solution.
\begin{lemma}\label{lem:eps:reg}
For each $\eps \in (0,1]$, we have 
\begin{align*}
u_{\eps} \in L^{2}(0,T;H^{2}(\Omega)), \quad u_{\Gamma,\eps} \in L^{2}(0,T;H^{2}(\Gamma)).
\end{align*}
\end{lemma}

\begin{proof}
Take a test function $\zeta \in C^{\infty}_{0}(\Omega)$ (so that $\zeta_{\Gamma} = 0$) in \eqref{eps:weakform} we see that $u_{\eps}$ satisfies
\begin{align*}
\Lap u_{\eps}(t) = \pd_{t} u_{\eps}(t) + \eps u_{\eps}(t) + \beta_{\eps}(u_{\eps}(t)) + \pi(u_{\eps}(t)) - f(t)
\end{align*}
in the sense of distributions for a.e.~$t \in (0,T)$.  By \eqref{ass:f}, Lipschitz continuity of $\beta_{\eps}$ and $\pi$, and the regularity of $u_{\eps}$ stated in Lemma \ref{lem:eps:wellpose}, the right-hand side belongs to $L^{2}(0,T;L^{2}(\Omega))$, and so $\Lap u_{\eps} \in L^{2}(0,T;L^{2}(\Omega))$.  Furthermore, the trace $u_{\Gamma,\eps}$ of $u_{\eps}$ on $\Sigma$ belongs to $L^{\infty}(0,T;H^{1}(\Gamma))$, and so by elliptic regularity \cite[Theorem 3.2, p.~1.79]{Brezzi} (see also Theorem \ref{Brezzi:reg}) we obtain
\begin{align*}
u_{\eps} \in L^{2}(0,T;H^{\frac{3}{2}}(\Omega)).
\end{align*}
Then, using the above estimate in conjunction with the fact that $\Lap u_{\eps} \in L^{2}(0,T;L^{2}(\Omega))$, we have by a variant of the trace theorem \cite[Theorem 2.27, p.~1.64]{Brezzi} (see also Theorem \ref{Brezzi:trace}) that $\pdnu u_{\eps} \in L^{2}(0,T;L^{2}(\Gamma))$, and from \eqref{eps:weakform} we obtain the following characterization for the surface part:
\begin{align*}
\LB u_{\Gamma,\eps}(t) = \pd_{t} u_{\Gamma,\eps}(t) +  \eps u_{\Gamma,\eps}(t) + \alpha \left (\beta_{\Gamma,\eps}(g(u_{\Gamma,\eps}(t))) + \pi_{\Gamma}(g(u_{\Gamma,\eps}(t))) - f_{\Gamma} \right ) + \alpha^{2} \pdnu u_{\eps}(t),
\end{align*}
holding for a.e.~$t \in (0,T)$.  As the right-hand side belongs to $L^{2}(0,T;L^{2}(\Gamma))$ we have by elliptic regularity that $u_{\Gamma,\eps} \in L^{2}(0,T;H^{2}(\Gamma))$.  This implies that the trace of $u_{\eps}$ belongs to $L^{2}(0,T;H^{\frac{3}{2}}(\Gamma))$ and by elliptic regularity we obtain
\begin{align*}
u_{\eps} \in L^{2}(0,T;H^{2}(\Omega)).
\end{align*}
\end{proof}

\subsubsection{Uniform estimates}

By virtue of Lemma \ref{lem:eps:reg} the approximation problem can be expressed as
\begin{subequations}\label{CF:approx:eps}
\begin{alignat}{3}
\pd_{t} u_{\eps} - \Lap u_{\eps} + \beta_{\eps}(u_{\eps}) + \eps u_{\eps} + \pi(u_{\eps}) = f & \text{ in } Q, \label{eps:bulk}  \\
\pd_{t} u_{\Gamma,\eps} - \LB u_{\Gamma,\eps} + \eps u_{\Gamma,\eps}  + \alpha ( \beta_{\Gamma,\eps}(g(u_{\Gamma,\eps})) + \pi_{\Gamma}(g(u_{\Gamma,\eps})) ) + \alpha^{2} \pdnu u_{\eps} = \alpha f_{\Gamma} & \text{ on } \Sigma, \label{eps:surf} \\
u_{\Gamma,\eps} = u_{\eps} \vert_{\Sigma} & \text{ on } \Sigma, \label{eps:Diri} \\
u_{\eps}(0) = u_{0} \text{ in } \Omega, \quad u_{\Gamma,\eps}(0) = u_{0} \vert_{\Gamma} & \text{ on } \Gamma, \label{eps:ini}
\end{alignat}
\end{subequations}
where \eqref{eps:bulk} and \eqref{eps:surf} hold pointwise a.e.~in $Q$ and on $\Sigma$, respectively.  We now derive estimates that are independent of $\eps$.  Below the symbol $C$ will denote positive constants that are independent of $\eps$.

\paragraph{First estimate.} Testing \eqref{eps:bulk} with $\pd_{t} u_{\eps}$, and using \eqref{eps:surf} leads to
\begin{align*}
& \frac{\dd}{\dt} \left ( \int_{\Omega} \frac{1}{2} \abs{\nabla u_{\eps}}^{2} + \hat{\beta}_{\eps}(u_{\eps}) + \hat{\pi}(u_{\eps}) + \frac{\eps}{2} \abs{u_{\eps}}^{2} \dx \right ) \\
& \qquad + \frac{\dd}{\dt} \left (  \int_{\Gamma} \frac{1}{2} \alpha^{-2} \abs{\surf u_{\Gamma,\eps}}^{2} + \hat{\beta}_{\Gamma, \eps}(g(u_{\Gamma,\eps})) + \hat{\pi}_{\Gamma}(g(u_{\Gamma,\eps})) + \frac{\eps}{2 \alpha^2} \abs{u_{\Gamma,\eps}}^{2} \dHaus  \right ) \\
& \qquad + \int_{\Omega} \abs{\pd_{t} u_{\eps}}^{2} \dx + \int_{\Gamma} \alpha^{-2} \abs{\pd_{t} u_{\Gamma,\eps}}^{2} \dHaus \\
 & \quad = \int_{\Omega} f \, \pd_{t} u_{\eps} \dx +\int_{\Gamma} \alpha^{-1} f_{\Gamma} \, \pd_{t} u_{\Gamma,\eps} \dHaus.
\end{align*}
In the above we have used that $\alpha^{-1} = g'(u_{\Gamma,\eps})$ and so
\begin{align*}
\int_{\Gamma} \alpha^{-1} \beta_{\Gamma,\eps}(g(u_{\Gamma,\eps})) \, \pd_{t} u_{\Gamma,\eps} \dHaus = \int_{\Gamma} \beta_{\Gamma,\eps}(g(u_{\Gamma,\eps})) \, \pd_{t} g(u_{\Gamma,\eps}) \dHaus = \frac{\dd}{\dt} \int_{\Gamma} \hat{\beta}_{\Gamma,\eps}(g(u_{\Gamma,\eps})) \dHaus.
\end{align*}
Applying Young's inequality and \eqref{Yosida:ini} we find that
\begin{equation}\label{CF:eps:1a}
\begin{aligned}
& \norm{\nabla u_{\eps}}_{L^{\infty}(0,T;L^{2}(\Omega))} + \norm{\hat{\beta}_{\eps}(u_{\eps}) + \hat{\pi}(u_{\eps})}_{L^{\infty}(0,T;L^{1}(\Omega))} \\
& \quad + \norm{\surf u_{\Gamma,\eps}}_{L^{\infty}(0,T;L^{2}(\Gamma))} + \norm{\hat{\beta}_{\Gamma,\eps}(g(u_{\Gamma,\eps})) + \hat{\pi}_{\Gamma}(g(u_{\Gamma,\eps}))}_{L^{\infty}(0,T;L^{1}(\Gamma))} \\
& \quad + \norm{\pd_{t} u_{\eps}}_{L^{2}(0,T;L^{2}(\Omega))}+ \norm{\pd_{t} u_{\Gamma,\eps}}_{L^{2}(0,T;L^{2}(\Gamma))} \leq C.
\end{aligned}
\end{equation}
Using \eqref{Linfty:L2} and a Gronwall argument we also have that
\begin{align}\label{CF:eps:1b}
\norm{u_{\eps}}_{L^{\infty}(0,T;L^{2}(\Omega))} + \norm{u_{\Gamma,\eps}}_{L^{\infty}(0,T;L^{2}(\Gamma))} \leq C.
\end{align}

\paragraph{Second estimate.} Lipschitz continuity of $\beta_{\eps}$ implies $\beta_{\eps}(u_{\eps}) \in L^{2}(0,T;H^{1}(\Omega))$, and so testing \eqref{eps:bulk} with $\beta_{\eps}(u_{\eps})$ and simplifying with \eqref{eps:surf} then yields
\begin{equation}
\label{CF:eps:3:est1}
\begin{aligned}
& \int_{\Omega} \underbrace{\beta_{\eps}'(u_{\eps}) \abs{\nabla u_{\eps}}^{2}}_{\geq 0} + \abs{\beta_{\eps}(u_{\eps})}^{2} + \underbrace{\eps u_{\eps} \beta_{\eps}(u_{\eps})}_{\geq 0} \dx + \int_{\Gamma} \underbrace{\alpha^{-2} \beta_{\eps}'(u_{\Gamma,\eps}) \abs{\surf u_{\Gamma,\eps}}^{2}}_{\geq 0} \dHaus \\
& \quad = - \int_{\Gamma} \alpha^{-1} \left ( \alpha^{-1} \pd_{t}u_{\Gamma,\eps} + \beta_{\Gamma,\eps}(g(u_{\Gamma,\eps})) + \alpha^{-1} \eps u_{\Gamma,\eps} + \pi_{\Gamma}(g(u_{\Gamma,\eps})) - f_{\Gamma} \right ) \beta_{\eps}(u_{\Gamma,\eps}) \dHaus \\
& \qquad - \int_{\Omega} (\pd_{t}u_{\eps} + \pi(u_{\eps}) - f) \beta_{\eps}(u_{\eps}) \dx.
\end{aligned}
\end{equation}
Recalling the resolvent operator $\mathcal{J}_{\eps} := (I - \eps \beta)^{-1}$ is a Lipschitz operator with constant $1$ and the fact that $\mathcal{J}_{\eps}(0) = 0 - \eps \beta_{\eps}(0) = 0$, from \eqref{resolvent} and the assumption \eqref{Limit:strong:beta} we have
\begin{align*}
\abs{\beta_{\eps}(u_{\Gamma,\eps})} \leq C \left ( 1 + \abs{ \mathcal{J}_{\eps}(u_{\Gamma,\eps})}^{q} \right ) \leq C(q) \left ( 1 + \abs{u_{\Gamma,\eps}}^{q} \right ).
\end{align*}
A similar estimate also holds for $\beta_{\Gamma,\eps}(g(u_{\Gamma,\eps}))$, namely
\begin{align*}
\abs{\beta_{\Gamma,\eps}(g(u_{\Gamma,\eps}))} \leq C(r, \alpha, \eta) \left ( 1 + \abs{u_{\Gamma,\eps}}^{r} \right ),
\end{align*}
and hence, the product term $\beta_{\Gamma,\eps}(g(u_{\Gamma,\eps})) \beta_{\eps}(u_{\Gamma,\eps})$ can be handled as follows:
\begin{align*}
-\int_{\Gamma} \alpha^{-1} \beta_{\Gamma,\eps}(g(u_{\Gamma,\eps})) \beta_{\eps}(u_{\Gamma,\eps}) \dHaus \leq C \int_{\Gamma} 1 + \abs{u_{\Gamma,\eps}}^{qr} \dHaus \leq C \left ( 1 + \norm{u_{\Gamma,\eps}}_{L^{\infty}(0,T;H^{1}(\Gamma))}^{qr} \right ) \leq C,
\end{align*}
thanks to the Sobolev embedding $H^{1}(\Gamma) \subset L^{s}(\Gamma)$ for any $s < \infty$.  Then, using \eqref{CF:eps:1a} and \eqref{CF:eps:1b} the right-hand side of \eqref{CF:eps:3:est1} can be estimated by
\begin{align*}
C \left ( 1 + \norm{\pd_{t} u_{\eps}}_{L^{2}(\Omega)}^{2} + \norm{\pd_{t} u_{\Gamma,\eps}}_{L^{2}(\Gamma)}^{2} + \norm{f_{\Gamma}}_{L^{2}(\Gamma)}^{2} + \norm{f}_{L^{2}(\Omega)}^{2} \right ) + \frac{1}{2} \norm{\beta_{\eps}(u_{\eps})}_{L^{2}(\Omega)}^{2},
\end{align*}
so that we obtain from \eqref{CF:eps:3:est1} the estimate
\begin{align}\label{CF:eps:3a}
\norm{\beta_{\eps}(u_{\eps})}_{L^{2}(0,T;L^{2}(\Omega))} \leq C,
\end{align}
and in turn via a comparison of terms in \eqref{eps:bulk} we have 
\begin{align*}
\norm{\Lap u_{\eps}}_{L^{2}(0,T;L^{2}(\Omega))} \leq C.
\end{align*}
Thanks to \eqref{CF:eps:3a}, in viewing \eqref{eps:bulk} as an elliptic equation for $u_{\eps}$ with right-hand side belonging to $L^{2}(0,T;L^{2}(\Omega))$ and Dirichlet boundary data $u_{\Gamma,\eps} \in H^{1}(\Gamma)$, by virtue of elliptic regularity and Theorem \ref{Brezzi:trace} one obtains
\begin{align}\label{CF:eps:3b}
\norm{u_{\eps}}_{L^{2}(0,T;H^{\frac{3}{2}}(\Omega))} + \norm{\pdnu u_{\eps}}_{L^{2}(0,T;L^{2}(\Gamma))} \leq C.
\end{align}

\begin{remark}\label{rem:CalaColli}
Note that for $g(s) = \alpha^{-1}(s-\eta)$, it may be natural to consider an assumption of the form
\begin{align*}
D(\beta) \supset D(\beta_{\Gamma} \circ g), \quad \abs{\beta^{\circ}(s)} \leq c_{0} \abs{\beta_{\Gamma}^{\circ}(g(s))} + c_{1} \quad \forall s \in \R
\end{align*}
which is analogous to \cite[(2.22)-(2.23)]{CalaColli}.  However, the above conditions
may not hold even in the case where $\beta = \beta_{\Gamma}$.  Take for example $\hat{\beta}(r) = I_{[-1,1]}(r)$ as the indicator function of the set $[-1,1]$, then it is well-known that
\begin{align*}
\beta(r) = \begin{cases}
[0,\infty) & \text{ for } r = 1, \\
0 & \text{ for } \abs{r} < 1, \\
(-\infty,0] & \text{ for } r = -1,
\end{cases} \quad \text{ with } D(\beta) = [-1,1],
\end{align*}
and a short computation shows that $D(\beta \circ g) = [-\alpha + \eta, \alpha + \eta]$ if $\alpha > 0$ and $D(\beta \circ g) = [\alpha + \eta, -\alpha + \eta]$ if $\alpha < 0$.  Hence, it is possible that $D(\beta) \cap D(\beta \circ g) = \emptyset$ for some values of $\alpha$ and $\eta$.
\end{remark}

\paragraph{Third estimate.}
Testing \eqref{eps:surf} with $g'(u_{\Gamma,\eps}) \beta_{\Gamma,\eps}(g(u_{\Gamma,\eps})) = \alpha^{-1} \beta_{\Gamma,\eps}(g(u_{\Gamma,\eps}))$ we have
\begin{equation}\label{CF:eps:3c}
\begin{aligned}
& \int_{\Gamma} \alpha^{-2} \beta_{\Gamma,\eps}'(u_{\Gamma,\eps}) \abs{\surf u_{\Gamma,\eps}}^{2} + \abs{\beta_{\Gamma,\eps}(g(u_{\Gamma,\eps}))}^{2}  \dHaus \\
& \quad = \int_{\Gamma} \left ( f_{\Gamma} - \alpha^{-1} \pd_{t}u_{\Gamma,\eps} -\alpha^{-1} \eps u_{\Gamma,\eps} - \alpha \pdnu u_{\eps} - \pi_{\Gamma}(g(u_{\Gamma,\eps})) \right ) \beta_{\Gamma,\eps}(g(u_{\Gamma,\eps})) \dHaus.
\end{aligned}
\end{equation}
In light of the estimate \eqref{CF:eps:3b} on $\pdnu u_{\eps}$ as well as previous estimates, we obtain 
\begin{align}\label{CF:eps:4a}
\norm{\beta_{\Gamma,\eps}(g(u_{\Gamma,\eps}))}_{L^{2}(0,T;L^{2}(\Gamma))} \leq C.
\end{align}
Then, viewing \eqref{eps:surf} as an elliptic equation for $u_{\Gamma,\eps}$ with right-hand side belonging to $L^{2}(0,T;L^{2}(\Gamma))$, elliptic regularity yields
\begin{align*}
\norm{u_{\Gamma,\eps}}_{L^{2}(0,T;H^{2}(\Gamma))} \leq C.
\end{align*}
This improved regularity of the Dirichlet boundary value $u_{\Gamma,\eps}$ for \eqref{eps:bulk} viewed as an elliptic equation for $u_{\eps}$ then leads to 
\begin{align}
\label{CF:eps:4b}
\norm{u_{\eps}}_{L^{2}(0,T;H^{2}(\Omega))} \leq C.
\end{align}

\subsubsection{Passing to the limit $\eps \to 0$}\label{sec:ptlim}
Owning to the uniform estimates \eqref{CF:eps:1a}, \eqref{CF:eps:1b}, \eqref{CF:eps:3a}-\eqref{CF:eps:4b}, we deduce the existence of a non-relabelled subsequence and limit functions $(u, u_{\Gamma}, \xi, \xi_{\Gamma})$ such that
\begin{equation*}
\begin{alignedat}{3}
u_{\eps} & \to u &&\text{ weakly-* in } L^{2}(0,T;H^{2}(\Omega)) \cap H^{1}(0,T;L^{2}(\Omega)) \cap L^{\infty}(0,T;H^{1}(\Omega)), \\
u_{\Gamma,\eps} & \to u_{\Gamma} && \text{ weakly-* in } L^{2}(0,T;H^{2}(\Gamma)) \cap H^{1}(0,T;L^{2}(\Gamma)) \cap L^{\infty}(0,T;H^{1}(\Gamma)), \\
\beta_{\eps}(u_{\eps}) & \to \xi && \text{ weakly in } L^{2}(0,T;L^{2}(\Omega)), \\
\beta_{\Gamma,\eps}(g(u_{\Gamma,\eps})) & \to \xi_{\Gamma} && \text{ weakly in } L^{2}(0,T;L^{2}(\Gamma)), \\
u_{\eps} & \to u && \text{ strongly in } C^{0}([0,T];L^{q}(\Omega)) \text{ for } q < 6,\\ 
u_{\Gamma,\eps} & \to u_{\Gamma} && \text{ strongly in } C^{0}([0,T];L^{r}(\Gamma)) \text{ for } r < \infty.
\end{alignedat}
\end{equation*}
By \eqref{eps:Diri}, it holds that $u_{\Gamma} = u \vert_{\Sigma}$ a.e.~on $\Sigma$.  Furthermore, by the linearity of the trace operator and the weak convergence of $u_{\eps}$ in $L^{2}(0,T;H^{2}(\Omega))$, it holds $\pdnu u_{\eps} \to \pdnu u$ weakly in $L^{2}(0,T;H^{\frac{1}{2}}(\Gamma))$.  We also obtain the assertion $\xi \in \beta(u)$ a.e.~in $Q$ and $\xi_{\Gamma} \in \beta_{\Gamma}(g(u_{\Gamma}))$ thanks to the strong convergences of $u_{\eps}$ and $u_{\Gamma,\eps}$, and the affine linearity of $g$.  Thus, passing to the limit $\eps \to 0$ in \eqref{CF:approx:eps} leads to \eqref{ACAC:limit}.
 
\subsubsection{Further regularity}
The regularity assertions 
\begin{align*}
\pd_{t} u \in L^{\infty}(0,T;L^2(\Omega)) \cap L^2(0,T;H^1(\Omega)), \quad \pd_{t} u_{\Gamma} \in L^{\infty}(0,T;L^2(\Gamma)) \cap L^2(0,T;H^1(\Gamma))
\end{align*}
can be derived formally by differentiating \eqref{eps:bulk} in time, testing the resulting equation by $\pd_{t} u_{\eps}$, integrating by parts and simplifying with \eqref{eps:surf}, leading to an estimate of the form
\begin{equation}\label{reg:1}
\begin{aligned}
& \frac{d}{dt} \frac{1}{2} \Big ( \norm{\pd_{t} u_{\eps}}_{L^{2}(\Omega)}^{2} + \alpha^{-2} \norm{\pd_{t} u_{\Gamma,\eps}}_{L^{2}(\Gamma)}^{2} \Big ) + \norm{\nabla \pd_{t} u_{\eps}}_{L^2(\Omega)}^{2} + \alpha^{-2} \norm{\surf \pd_{t} u_{\Gamma,\eps}}_{L^2(\Gamma)}^2 \\
& \qquad + \int_{\Omega} \beta_{\eps}'(u_{\eps}) \abs{\pd_{t} u_{\eps}}^2 + \eps \abs{\pd_{t} u_{\eps}}^2 \dx + \int_{\Gamma} \alpha^{-2} \beta'_{\Gamma,\eps}(g(u_{\Gamma,\eps})) \abs{\pd_{t} u_{\Gamma,\eps}}^2 \dHaus \\
& \quad \leq C \Big (\norm{\pd_{t} u_{\eps}}_{L^2(\Omega)}^2 + \norm{\pd_{t} u_{\Gamma,\eps}}_{L^2(\Gamma)}^2 \Big ) + C \Big (\norm{\pd_{t} f}_{L^2(\Omega)}^2 + \norm{\pd_{t} f_{\Gamma}}_{L^2(\Gamma)}^2 \Big ),
\end{aligned} 
\end{equation}
where we used the Lipschitz continuity of $\pi$ and $\pi_{\Gamma}$.  Note that
\begin{align*}
\pd_{t} u_{\eps}(0) & := \Lap u_{0} - \beta_{\eps}(u_{0}) - \pi(u_{0}) - \eps u_{0} + f(0), \\
\pd_{t} u_{\Gamma,\eps}(0) & := \LB u_{\Gamma,0} - \alpha (\beta_{\Gamma,\eps}(g(u_{\Gamma,0})) - \pi_{\Gamma}(g(u_{\Gamma,0})) - \eps u_{\Gamma,0} ) - \alpha^2 \pdnu u_0 + \alpha f_{\Gamma}(0),
\end{align*}
where $u_{\Gamma,0} := u_0 \vert_{\Gamma}$, are bounded uniformly in $L^2(\Omega)$ and $L^2(\Gamma)$, respectively.  Indeed, we use the property $\beta_\eps(v) \rightharpoonup \beta^{\circ}(v)$ in $L^2(\Omega)$ as $\eps \to 0$ for any $v \in D(\beta)$ when viewing $\beta:L^2(\Omega) \to 2^{L^2(\Omega)}$ as a set-valued operator.  Then, the boundedness assertion for $\pd_{t} u_{\eps}(0)$ is a consequence of the assumption \eqref{ass:ini:strong}.  A similar reasoning is applied to show the boundedness of $\pd_{t} u_{\Gamma,\eps}(0)$.  Therefore, applying \eqref{ass:f}, the non-negativity of $\beta'_{\eps}$ and $\beta'_{\Gamma,\eps}$ and the Gronwall inequality, we infer (formally) from \eqref{reg:1} that 
\begin{align*}
\norm{\pd_{t} u_{\eps}}_{L^{\infty}(0,T;L^2(\Omega)) \cap L^2(0,T;H^1(\Omega))} + \norm{\pd_{t} u_{\Gamma,\eps}}_{L^{\infty}(0,T;L^2(\Gamma)) \cap L^2(0,T;H^1(\Gamma))} \leq C.
\end{align*}
Returning to \eqref{CF:eps:3:est1} we find that $\beta_{\eps}(u_{\eps})$ is bounded in $L^{\infty}(0,T;L^2(\Omega))$, and we improve the time regularity in \eqref{CF:eps:3b} from $L^2(0,T)$ to $L^{\infty}(0,T)$.  Then, in a similar fashion, from \eqref{CF:eps:3c} we now infer that $\beta_{\Gamma,\eps}(g(u_{\Gamma,\eps}))$ is bounded in $L^{\infty}(0,T;L^2(\Gamma))$, which yields that $u_{\Gamma,\eps}$ is bounded in $L^{\infty}(0,T;H^2(\Gamma))$ and $u_{\eps}$ is bounded in $L^{\infty}(0,T;H^2(\Omega))$.

However, there is no indication from the proof of Lemma \ref{lem:eps:wellpose} that we can differentiate \eqref{eps:bulk} in time, since the second time derivative of $\bm{u}_{\eps}$ need not exist.  To justify the above formal computations in a rigorous way, we can employ a Faedo--Galerkin approximation for \eqref{CF:approx:eps}, similar as in \cite{CalaColli}, and perform the same procedure as in Section \ref{sec:UniEst} for the Robin problem \eqref{Intro:ACAC:gen} at the Galerkin level.  This yields an analogue of \eqref{reg:1} for the Galerkin solutions, similar to \eqref{Est:2a} and \eqref{Est:2b}, so that after passing to the limit $n \to \infty$ and $\eps \to 0$, we obtain a strong solution $(\tilde u, \tilde u_{\Gamma}, \tilde \xi, \tilde \xi_{\Gamma})$ with the regularity as stated in Theorem \ref{thm:Limit:Exist}.  Then, in conjunction with the uniqueness of solutions to \eqref{Limit:form:h} (see Theorem \ref{thm:Limit:ctsdep}), we see that the limit solution $(\tilde u, \tilde u_{\Gamma}, \tilde \xi, \tilde \xi_{\Gamma})$ from the Faedo--Galerkin approximation must coincide with the limit solution $(u, u_{\Gamma}, \xi, \xi_{\Gamma})$ obtained in Section \ref{sec:ptlim}, and so it holds that $(u, u_{\Gamma}, \xi, \xi_{\Gamma})$ also satisfies the regularity stated in Theorem \ref{thm:Limit:Exist}.

\subsection{Error estimates}
For $K > 0$, let $(u_{K}, \phi_{K})$ denote the unique strong solution to \eqref{ACAC:gen} with corresponding data $(u_{0,K}, \phi_{0,K}, f_{K}, f_{\Gamma,K})$, and let $(u,\phi)$ denote the unique strong solution to \eqref{Limit:form:h} with corresponding data $(u_{0}, \phi_{0}, f, f_{\Gamma})$.  Let $\hat{u} := u_{K} - u$, $\hat{\phi} := \phi_{K} - \phi$, $\hat{\xi} := \xi_{K} - \xi$, $\hat{\xi}_{\Gamma} := \xi_{\Gamma,K} - \xi_{\Gamma}$, $\hat{f} := f_{K} - f$ and $\hat{f}_{\Gamma} := f_{\Gamma,K} - f_{\Gamma}$ denote their differences and consider an arbitrary test function $\zeta \in H^{1}(\Omega)$.  From the bulk equations we see that
\begin{align}\label{Rate:bulk}
\int_{\Omega} \pd_{t}\hat{u} \, \zeta + \nabla \hat{u} \cdot \nabla \zeta + \hat{\xi} \, \zeta + (\pi(u_{K}) - \pi(u)) \, \zeta - \hat{f} \, \zeta \dx - \int_{\Gamma} \pdnu \hat{u} \, \zeta \dHaus = 0,
\end{align}
and from the surface equations we obtain for an arbitrary test function $\theta \in H^{1}(\Gamma)$,
\begin{align}\label{Rate:surf}
\int_{\Gamma} \frac{1}{\alpha} \left (\pd_{t} \hat{\phi} \, \theta + \surf \hat{\phi} \cdot \surf \theta + \hat{\xi}_{\Gamma} \, \theta + (\pi_{\Gamma}(\phi_{K}) - \pi_{\Gamma}(\phi)) \, \theta - \hat{f}_{\Gamma} \, \theta \right ) + \pdnu \hat{u} \, \theta \dHaus = 0.
\end{align}
Using the fact that $u \vert_{\Sigma} = h(\phi) = \alpha \phi + \eta$, a short calculation shows that
\begin{align*}
\pdnu \hat{u} = \pdnu (u_{K} - u) = \frac{1}{K}( h(\phi_{K}) - u_{K}) - \pdnu u = \frac{1}{K} (\alpha \hat{\phi} - \hat{u}) - \pdnu u.
\end{align*}
Substituting the above into \eqref{Rate:bulk} and \eqref{Rate:surf}, then substituting $\zeta = \hat{u}$ and $\theta = \alpha \hat{\phi}$ and adding the two equalities leads to
\begin{equation}\label{Rate:Cal1}
\begin{aligned}
& \frac{1}{2} \frac{\dd}{\dt} \left ( \norm{\hat{u}}_{L^{2}(\Omega)}^{2} + \norm{\hat{\phi}}_{L^{2}(\Gamma)}^{2} \right ) + \norm{\nabla \hat{u}}_{L^{2}(\Omega)}^{2} + \norm{\surf \hat{\phi}}_{L^{2}(\Gamma)}^{2} + \frac{1}{K} \norm{\alpha \hat{\phi} - \hat{u}}_{L^{2}(\Gamma)}^{2} \\
& \quad \leq L_{\pi} \norm{\hat{u}}_{L^{2}(\Omega)}^{2} + L_{\pi_{\Gamma}} \norm{\hat{\phi}}_{L^{2}(\Gamma)}^{2} + \norm{\hat{f}}_{L^{2}(\Omega)} \norm{\hat{u}}_{L^{2}(\Omega)} + \norm{\hat{f}_{\Gamma}}_{L^{2}(\Gamma)} \norm{\hat{\phi}}_{L^{2}(\Gamma)} \\
& \qquad + \norm{\pdnu u}_{L^{2}(\Gamma)} \norm{\alpha \hat{\phi} - \hat{u}}_{L^{2}(\Gamma)},
\end{aligned}
\end{equation}
where we have used the monotonicity of $\beta$ and $\beta_{\Gamma}$, and the Lipschitz continuity of $\pi$ and $\pi_{\Gamma}$.  Using Young's inequality, the right-hand side of \eqref{Rate:Cal1} can be estimated by
\begin{align*}
&\frac{1}{2K} \norm{\alpha \hat{\phi} - \hat{u}}_{L^{2}(\Gamma)}^{2} + \frac{K}{2} \norm{\pdnu u}_{L^{2}(\Gamma)}^{2} \\
& \quad + \frac{1}{2} \left ( (2 L_{\pi} + 1) \norm{\hat{u}}_{L^{2}(\Omega)}^{2} + (2 L_{\pi_{\Gamma}} + 1) \norm{\hat{\phi}}_{L^{2}(\Gamma)}^{2}  + \norm{\hat{f}}_{L^{2}(\Omega)}^{2} + \norm{\hat{f}_{\Gamma}}_{L^{2}(\Gamma)}^{2} \right ),
\end{align*}
so that by Gronwall's inequality, we have
\begin{align*}
& \norm{\hat{u}}_{L^{\infty}(0,T;L^{2}(\Omega)) \cap L^{2}(0,T;H^{1}(\Omega))}^{2} + \norm{\hat{\phi}}_{L^{\infty}(0,T;L^{2}(\Gamma)) \cap L^{2}(0,T;H^{1}(\Gamma))}^{2} + \frac{1}{K} \norm{\alpha \hat{\phi} - \hat{u}}_{L^{2}(\Sigma)}^{2} \\
& \quad \leq C \left ( \norm{\hat{f}}_{L^{2}(Q)}^{2} + \norm{\hat{f}_{\Gamma}}_{L^{2}(\Sigma)}^{2} + \norm{\hat{u}_{0}}_{L^{2}(\Omega)}^{2} + \norm{\hat{\phi}_{0}}_{L^{2}(\Gamma)}^{2} + K \norm{\pdnu u}_{L^{2}(\Gamma)}^{2} \right ),
\end{align*}
which is the inequality \eqref{Rate:Est:1}.

\appendix
\section{Trace theorem and elliptic regularity}
For the convenience of the reader, in this section we state the results of the trace theorem \cite[Theorem 2.27, p.~1.64]{Brezzi} we use to deduce that the normal derivative belongs to $L^{2}(0,T;L^{2}(\Gamma))$ and the elliptic regularity result \cite[Theorem 3.2, p.~1.79]{Brezzi} which we have used extensively in this work.

\begin{thm}[{\cite[Theorem 2.27]{Brezzi}}]\label{Brezzi:trace}
Let $\Omega$ be a smooth domain or a half space with boundary $\Gamma$, and let $A$ be an elliptic operator of the form
\begin{align}\label{form:A}
Av = - \sum_{i,j=1}^{d} \pd_{x_j} (a_{ij} \pd_{x_i} v) + \sum_{i=1}^{d} b_{i} \pd_{x_{i}} v - \sum_{i=1}^{d} \pd_{x_{i}}(c_i v) + dv
\end{align}
with coefficients $a_{ij}, b_i, c_i, d \in C^{\infty}(\overline{\Omega})$ and the matrix $(a_{ij})_{1 \leq i,j \leq d}$ is positive-definite uniformly in $x \in \Omega$.  Consider the space defined with the graph norm
\begin{align*}
W_{A,k}^{s,p}(\Omega) := \{ v \in W^{s,p}(\Omega) \, : \, Av \in W^{-2+k+\frac{1}{p}, p}(\Omega) \} \text{ for } k = 0,1.
\end{align*}
Then, for $s \in \R$, $1 < p < \infty$ and either $p = 2$ or $s - \frac{1}{p} \nin \Z$, there exist unique operators
\begin{align*}
\gamma_0 \in \mathcal{L}(W^{s,p}_{A,0}(\Omega) \, ; \, W^{s-\frac{1}{p}, p}(\Gamma)), \text{ and } \gamma_A \in \mathcal{L}(W^{s,p}_{A,1}(\Omega) \, ; \, W^{s-1-\frac{1}{p},p}(\Gamma))
\end{align*}
such that $\gamma_0 v = v \vert_\Gamma$ and $\gamma_A v = \frac{\pd v}{\pd \bm{\nu}_A} \vert_\Gamma := (\sum_{i,j=1}^{d} a_{ij} \nu_j \pd_{x_i} v + \sum_{i=1}^d c_i \nu_j v) \vert_{\Gamma}$ for all $v \in C^{\infty}(\overline{\Omega}) \cap W^{s,p}_{A,j}(\Omega)$.  Furthermore, there exists an operator
\begin{align*}
R \in \mathcal{L}( W^{s-\frac{1}{p}, p}(\Gamma) \times W^{s-1-\frac{1}{p}, p}(\Gamma) \, ; \, W^{s,p}_{A,1}(\Omega))
\end{align*}
such that $\gamma_0 R(g_0, g_1) = g_0$ and $\gamma_A R(g_0, g_1) = g_1$ for all $g_k \in W^{s-k-\frac{1}{p},p}(\Gamma)$ for $k = 0,1$.
\end{thm}
For our purpose in the proof of Lemma \ref{lem:eps:reg}, we employ Theorem \ref{Brezzi:trace} with $a_{ij} = \delta_{ij}$, $b_i = c_i = d = 0$, $s = \frac{3}{2}$, $p = 2$ and $k = 1$ with $u_{\eps} \in L^{\infty}(0,T;H^{\frac{3}{2}}(\Omega))$ and $\Lap u_{\eps} \in L^{2}(0,T;L^{2}(\Omega))$ to deduce that $\gamma_A u_\eps = \pdnu u_\eps$ belongs to $L^{2}(0,T;L^{2}(\Gamma))$.

\begin{thm}[{\cite[Theorem 3.2]{Brezzi}}]\label{Brezzi:reg}
Let $\Omega$ be a smooth and bounded domain with boundary $\Gamma$, and let $A$ be a uniformly elliptic operator of the form \eqref{form:A}.  Assume $u$ is a distributional solution to the equation $Au = f$ in $\Omega$ with one of the following boundary conditions
\begin{align*}
u \vert_\Gamma = g_0 \text{ or } \frac{\pd u }{\pd \bm{\nu}_A} \vert_\Gamma = g_1.
\end{align*}
Assume $r, t, s \in \R$, $1 < p < \infty$, and either $p = 2$ or $s - \frac{1}{p} \nin \Z$, and let $f \in W^{r,p}(\Omega)$.  If $g_0 \in W^{t,p}(\Gamma)$, $r+2 \geq \frac{1}{p}$ and $s = \min(r + 2, t + \frac{1}{p})$, then the solution $u$ belongs to $W^{s,p}(\Omega)$.  If $g_1 \in W^{t,p}(\Gamma)$, $r+1 \geq \frac{1}{p}$ and $s = \min(r + 2, t + 1 + \frac{1}{p})$, then the solution $u$ belongs to $W^{s,p}(\Omega)$.  In both cases we have the estimate
\begin{align*}
\norm{u}_{W^{s,p}(\Omega)} \leq C \Big ( \norm{f}_{W^{r,p}(\Omega)} + \norm{g_j}_{W^{t,p}(\Gamma)} \Big ) \text{ for } j = 0, 1,
\end{align*}
with a constant $C$ not depending on $f$ and $g_j$.
\end{thm}

\section*{Acknowledgments}
K.F. Lam expresses his gratitude to P. Colli and the University of Pavia for their hospitality during his visit in which part of this research was completed, and gratefully acknowledges support from a Direct Grant of CUHK (project 4053288).  On the other hand, P. Colli gratefully acknowledges partial support from the MIUR-PRIN Grant 2015PA5MP7 ``Calculus of Variations'', the GNAMPA (Gruppo Nazionale per l'Analisi Matematica, la Probabilit\`{a} e le loro Applicazioni) of INdAM (Istituto Nazionale di Alta Matematica) and the IMATI -- C.N.R. Pavia.  Moreover, T. Fukao gratefully acknowledges the support from the JSPS KAKENHI Grant-in-Aid for Scientific Research(C), Grant Number 17K05321.  The authors would like to thank the anonymous referees for their careful reading and useful suggestions which have improved the quality of the manuscript.

\bibliographystyle{plain}
\bibliography{CFL}
\end{document}